\RequirePackage[l2tabu,orthodox]{nag}
\documentclass[reqno,12pt,a4paper,oneside]{amsart}
\usepackage
           {geometry}
\usepackage{ifxetex,ifluatex}
\newif\ifxetexorluatex
\ifxetex
  \xetexorluatextrue
\else
  \ifluatex
    \xetexorluatextrue
  \else
    \xetexorluatexfalse
  \fi
\fi

\ifxetexorluatex
  \usepackage{fontspec}           
  \usepackage{unicode-math}       
  \usepackage{polyglossia}        
  \setmainlanguage{english}
\else
  \usepackage[utf8]{inputenc}     
  \usepackage[english]{babel}     

  \usepackage{amsfonts}
  \usepackage{amssymb}
  
  \usepackage{lmodern}
  \usepackage[T1]{fontenc}
\fi
\usepackage{booktabs}   
\usepackage[inline]{enumitem} 
\usepackage{xspace}     
\usepackage[svgnames]{xcolor}     
\usepackage{amsmath}    
\usepackage{mathtools}  
\usepackage{stmaryrd} 

\usepackage{graphics}
\usepackage{etoolbox}
\usepackage{tikz-cd}

\definecolor{linkblue}{RGB}{1,1,190}
\definecolor{citegreen}{RGB}{1,190,1}
\usepackage[linkcolor=linkblue
           ,citecolor=citegreen
           ,ocgcolorlinks        
           ,bookmarksopen=True,  
           ]{hyperref}
\usepackage[ocgcolorlinks]{ocgx2}
\makeatletter
  \AtBeginDocument{
    \hypersetup{
      pdftitle  = {\@title},
      pdfauthor = {\authors}     
    }
  }
\makeatother

\usepackage{amsthm}

\usepackage[capitalize    
           ]{cleveref}

\theoremstyle{definition}
\newtheorem {definition}{Definition}[section]

\theoremstyle{plain}
\newtheorem {theorem}[definition]{Theorem}
\crefname   {theorem}{Theorem}{Theorems}
\newtheorem*{theorem*}{Theorem}
\newtheorem {lemma}[definition]{Lemma}
\newtheorem {proposition}[definition]{Proposition}

\newtheorem {corollary}[definition]{Corollary}

\newtheorem {notation}[definition]{Notation and Definition}
\newtheorem {questions}[definition]{Questions}
\newtheorem {question}[definition]{Question}

\theoremstyle{remark}
\newtheorem {remark}[definition]{Remark}
\newtheorem {remarks}[definition]{Remarks}
\newtheorem {example}[definition]{Example}
\newtheorem*{example*}{Example}
\crefname   {example}{Example}{Examples}
\newtheorem {examples}[definition]{Examples}

\AtBeginEnvironment{smallremark}{\footnotesize}

\makeatletter

\newcommand{\sc@lettershortcut}[3]{%
  \expandafter\providecommand\csname #2#3\endcsname{#1{#3}}%
}

\newcommand{\sc@shortcuts}[3]{%
  \count@=0
  \loop
  \advance\count@ 1
  \edef\tmp@{%
    \noexpand\sc@lettershortcut\unexpanded{{#1}}{#2}{#3\count@}
  }
  \tmp@
  \ifnum\count@<26
  \repeat
}

\newcommand{\defshortcuts}[2]{\sc@shortcuts{#1}{#2}{\@alph}}

\newcommand{\defShortcuts}[2]{\sc@shortcuts{#1}{#2}{\@Alph}}

\makeatother
\defShortcuts{\mathbb}{b}
\defShortcuts{\mathcal}{c}
\defShortcuts{\mathfrak}{f}
\defShortcuts{\mathsf}{s}
\defshortcuts{\mathfrak}{f}
\defshortcuts{\mathsf}{s}

\def\rfop{*}
\makeatletter
\newcommand\rigidfactorization[2][]{%
  \def\rf@delim{\rfop}
  \newif\ifrf@notfirst
  #1
  \@for\next:=#2\do{%
    \ifrf@notfirst
      \rf@delim
    \fi
    \rf@notfirsttrue
    \next
  }%
}
\makeatother
\newcommand\rf\rigidfactorization
\ifxetexorluatex
  
\else
  
\fi
\providecommand{\val}{\mathsf{v}}                 

\newcommand{\cc}[2]{{({#1}\!:\!{#2})}}          
\newcommand{\lc}[2]{{({#1}\!\!:_l\!{#2})}}        
\newcommand{\rc}[2]{{({#1}\!\!:_r\!{#2})}}        

\DeclarePairedDelimiter{\card}{\lvert}{\rvert}

\DeclareMathOperator{\End}{End}

\DeclareMathOperator{\Hom}{Hom}
\DeclareMathOperator{\Ext}{Ext}

\DeclareMathOperator{\ann}{ann}

\DeclareMathOperator{\id}{id}

\DeclareMathOperator{\supp}{supp}
\DeclareMathOperator{\im}{im}

\DeclareMathOperator{\GKdim}{GKdim}
\DeclareMathOperator{\rKdim}{rKdim}

\DeclareMathOperator{\gr}{gr} 

\newcommand{\defit}[1]{\textsf{#1}}

\setlist[enumerate,1]{label=\textup{(\arabic*)}, ref=\textup{(}\arabic*\textup{)}, leftmargin=0.75cm}
\setlist[enumerate,2]{label=\textup{(}\roman*\textup{)}, ref=\textup{(}\roman*\textup{)}}

\newlist{equivenumerate}{enumerate}{1}
\setlist[equivenumerate,1]{%
  label=\textup{(\alph*)},
  ref=\textup{(}\alph*\textup{)},
  leftmargin=0.75cm
}

\newlist{equivenumerate*}{enumerate*}{1}
\setlist*[equivenumerate*,1]{%
  label=\textup{(\alph*)},
  ref=\textup{(}\alph*\textup{)},
  leftmargin=0.75cm
}

\newlist{propenumerate}{enumerate}{1}
\setlist[propenumerate,1]{%
  label=\textup{(\roman*)},
  ref=\textup{(}\roman*\textup{)},
  leftmargin=0.75cm
}

{\end{description}}

\makeatletter
\def\sr@stripleadingcol::#1{#1}
\def\sr@dosubref#1#2:#3 #4{\if\relax#3\relax%
  \def\first{\sr@stripleadingcol #4}%
  #1{\first}\ref{\first:#2}%
\else%
  \sr@dosubref#1#3 {#4:#2}%
\fi}%

\newcommand{\subref}[1]{\sr@dosubref\cref#1: :\relax}
\newcommand{\Subref}[1]{\sr@dosubref\Cref#1: :\relax}
\makeatother
\AtBeginDocument{\renewcommand{\setminus}{\smallsetminus}}

\usepackage{microtype}

\title{On noncommutative bounded factorization domains and prime rings}

\author{Jason P. Bell}
\address{Department of Pure Mathematics, University of Waterloo, Waterloo, ON, Canada N2L 3G1}
\email{jpbell@uwaterloo.ca}

\author{Ken Brown}
\address{School of Mathematics and Statistics, University of Glasgow, Glasgow G12 8QW, Scotland}
\email{ken.brown@glasgow.ac.uk}

\author{Zahra Nazemian}
\address{University of Graz\\
          NAWI Graz\\
          Institute for Mathematics and Scientific Computing\\
          Heinrichstra\ss e 36\\
          8010 Graz, Austria}
\email{zahra.nazemian@uni-graz.at}

\author{Daniel Smertnig}
\email{daniel.smertnig@uni-graz.at}

\ifxetexorluatex
  
\else
  
  \fi


\begin{document}

\begin{abstract}
  A ring has \emph{bounded factorizations} if every cancellative nonunit $a \in R$ can be written as a product of atoms and there is a bound $\lambda(a)$ on the lengths of such factorizations.
  The bounded factorization property is one of the most basic finiteness properties in the study of non-unique factorizations.
  Every commutative noetherian domain has bounded factorizations, but it is open whether such a result holds in the noncommutative setting.
  We provide sufficient conditions for a noncommutative noetherian prime ring to have bounded factorizations.
  Moreover, we construct a (noncommutative) finitely presented semigroup algebra that is an atomic domain but does not satisfy the ascending chain condition on principal right or left ideals (ACCP), whence it does not have bounded factorizations.
\end{abstract}

\maketitle

\section{Introduction}

Let $R$ be a ring, denote by $R^\bullet$ its submonoid of cancellative elements (non-zero-divisors), and by $R^\times$ its group of units.
Usually $R$ will be a domain or a prime Goldie ring, in which case $R^\bullet$ is divisor-closed, that is, every left- or right-divisor of an element of $R^\bullet$ is again contained in $R^\bullet$ (see \cref{l:cancellative} below). A nonunit $u \in R^\bullet$ is an \defit{atom} (or \defit{irreducible element}) if it cannot be written as a proper product of two nonunits in $R^\bullet$. We say that $R$ is \defit{atomic} if every nonunit $a \in R^\bullet$ can be expressed as a product $a=u_1\cdots u_k$ of atoms $u_1$, $\ldots\,$,~$u_k$ of $R^\bullet$.
The ring $R$ has \defit{bounded factorizations} (is a \defit{BF-ring}) if, in addition, for every $a \in R^\bullet$ there exists a $\lambda(a) \in \bN_0$ such that $k \le \lambda(a)$ for every such factorization of $a$.

The concept of bounded factorization domains was introduced, in the setting of commutative domains, by D.\,D.~Anderson, D.\,F.~Anderson, and M.~Zafrullah \cite{anderson-anderson-zafrullah90} and is one of the most basic finiteness notions in the study of non-unique factorizations (see the recent surveys \cite{geroldinger16,geroldinger-zhong20} and in particular \cite{anderson-gotti22}).
Our restriction to cancellative elements is, to degree, necessitated by the fact that every monoid having bounded factorizations is at least unit-cancellative (see \cref{l:length-basic} below).

Chain conditions on (one-sided) ideals are well-known to imply factorization-related properties: the ring $R$ satisfies the \defit{ascending chain condition on principal ideals} (or \defit{ACCP}) if every ascending chain of principal right ideals eventually stabilizes and the same is true for chains of principal left ideals.
Every ring satisfying the ACCP is atomic (in fact, it suffices to have the ACC on principal right, respectively left, ideals generated by cancellative elements).
In particular, noetherian rings are atomic.

If $R$ is a \emph{commutative} noetherian domain, then $R$ even has bounded factorizations. The standard proof of this fact can be found in any of \cite[Theorem 4.9]{anderson-gotti22}, \cite[Proposition 2.2]{anderson-anderson-zafrullah90}, or \cite[Corollary 1.3.5]{GHK06}.
More generally, every $v$-noetherian commutative cancellative monoid has bounded factorizations.
This can be proved analogously to the case of noetherian domains \cite[Theorem 2.2.9]{GHK06}.
A different proof can be obtained by first showing that every $v$-noetherian commutative cancellative monoid has finite $\omega$-invariant \cite[Theorem 4.2]{geroldinger-hassler08}, which immediately implies the claim \cite[Lemma 3.3(3)]{geroldinger-hassler08}.
These results raise the question whether such an implication still holds for noncommutative noetherian domains, or more generally, noncommutative noetherian prime rings.
The proofs of the commutative setting do not carry over to the noncommutative one, because all of them make use of localizations or prime ideals in ways that do not generalize.

In the present paper we therefore seek \emph{sufficient} conditions for a noncommutative noetherian prime ring to have bounded factorizations.
We show that a noetherian prime ring $R$ has bounded factorizations if it satisfies one of the following conditions:
\begin{enumerate}[leftmargin=1.5em]
\item $R$ has right (or left) Krull-dimension $\le 1$ in the sense of Gabriel--Rentschler (\cref{p:smalldim}).
\item $R$ has a filtration $R_0 \subseteq R_1 \subseteq \cdots$ with $R_0 \subseteq R^\times$ such that the associated graded ring is a domain (\cref{p:filtration}).
\item $R$ is an iterated skew [Laurent] polynomial domain over a commutative or BF-ring $S$  (\cref{p:skew,cor:iterated}).
\item $R$ is a bounded Krull order (in the sense of Marubayashi or Chamarie; \cref{p: bounded}).
\item\label{m:fbn} $R$ is fully bounded noetherian (FBN) and every nonzero two-sided ideal contains a nonzero central element (\cref{thm: FBN}).
  In particular, this holds when $R$ is a PI ring (\cref{c:prime-pi}), with a second proof in this case given in \cref{subsec:pi}.
\item\label{m:quadratic} $R$ is an affine algebra of quadratic growth (\cref{thm: growth}).
\item\label{m:ag} $R$ is an Auslander-Gorenstein ring (\cref{AG}).
  See \cref{8.6} for an extensive list of rings covered by this class.
\item\label{m:adual} $R$ is a prime quotient of a noetherian algebra $S$ over a field, with $S$ having an Auslander-dualizing complex (\cref{gradejump}).
\end{enumerate}

In the first four of these classes the proofs are largely straightforward.
For \ref{m:fbn} we make use of the reduced rank and a noncommutative version of the principal ideal theorem.
Extra conditions appearing in the principal ideal theorem are the cause for the restriction to those FBN rings whose nonzero two-sided ideals contain a nonzero central element.

For \ref{m:quadratic} we establish that, if $R$ is not a PI ring and $a \in R^\bullet$ is a nonunit, then $R/aR$ has, in a certain sense, linear growth.
The crucial part is to establish $\dim(R/aR) = \infty$.
For this, we make use of a theorem of J.\,P.~Bell and A.~Smoktunowicz about the extended center of such algebras \cite[Theorem 1.2]{bell-smoktunowicz10}, and a result of Martindale, asserting that such a ring does not satisfy a linear generalized polynomial identity (\cref{l:no-linear-gpi}).

Finally the proofs of \ref{m:ag} and \ref{m:adual} make use of homological methods.
For an Auslander-Gorenstein ring $R$, one can define a finitely partitive grade function $j$ on $R$-modules (see \cref{AusGor}).
Observing $j(R/aR)=1$ for all nonunits $a \in R^\bullet$, allows us to deduce that Auslander-Gorenstein rings are BF-rings.

While many large classes of important rings are Auslander-Gorenstein rings (such as group algebras of polycyclic-by-finite groups, all known noetherian Hopf algebras, etc., see \cref{8.6}), all Auslander-Gorenstein rings have finite injective dimension and finite Krull dimension. To overcome these restrictions, for algebras over a field, A.~Yekutieli and J.\,J.~Zhang developed the more general notion of Auslander dualizing complexes using the machinery of derived categories.
For an algebra $R$ having an Auslander dualizing complex one may again introduce a grade function (with $-j$ the canonical dimension of A.~Yekutieli and J.\,J.~Zhang).
Then $-j$ satisfies Gabber's Maximality Principle on $j$-pure $R$-modules.
Using these properties, we show that if $R$ is a noetherian $K$-algebra with an Auslander dualizing complex, and $R/I$ is a $j$-pure factor ring of $R$ with an artinian classical ring of quotients, then $R/I$ is a BF-ring (\cref{gradejump}). In particular, this applies to $R/P$ with $P$ a prime ideal of $R$.

We do not know an example of a noetherian prime ring that does not have bounded factorizations, and so, in a sense, the basic question, whether \emph{every} noetherian prime ring has bounded factorizations, unfortunately remains open.

A different point of view along one which might try to extend the commutative result is the following:
commutative affine domains over fields are noetherian and therefore have bounded factorizations.
Noncommutative finitely generated (or even finitely presented) algebras need not be noetherian, but one may still ask whether any finitely presented (atomic) prime ring has bounded factorizations.
In \cref{s:exm-non-bf} we construct a semigroup algebra $R$, such that $R$ is finitely presented over a field $K$ and such that $R$ is atomic but does not have bounded factorizations (and indeed does not even satisfy the ACCP).

It is well-known that every domain $R$ that satisfies the ACCP is atomic.
While the converse is not true, the difference is somewhat subtle. (P.\,M.~Cohn, in \cite{cohn68}, somewhat infamously falsely asserted that equivalence holds.)
The first counterexample, a commutative domain that is atomic but does not satisfy ACCP, was constructed by A.~Grams \cite{grams74}.
Further classic constructions are by A.~Zaks \cite{zaks82} and by M.~Roitman \cite{roitman93} (who in fact constructed an atomic domain $R$ such that the polynomial ring $R[X]$ is not atomic).
Nevertheless producing simple examples of atomic domains that do not satisfy the ACCP remains challenging, even in the commutative setting, with recent contributions by J.\,G.~Boynton and J.~Coykendall \cite{boynton-coykendall19}, as well as F.~Gotti and B.~Li \cite{gotti-li21,gotti-li22}.

The first example of a semigroup algebra over a field that is atomic but does not satisfy the ACCP is given in \cite{gotti-li21}.
Our construction in \cref{s:exm-non-bf} provides the first example of a \emph{finitely presented} semigroup algebra over a field that is atomic but does not satisfy the ACCP.
Of course, such an example is only possible in the noncommutative setting.
\cref{s:exm-non-bf} can be read largely independently of the rest of the paper.

\textbf{Acknowledgements.}
We thank M. Hochster and and R. Heitmann for providing us with the construction in \cref{melexample}.
The second author is grateful for the support of Leverhulme Emeritus Fellowship EM-2017-081.
The first author acknowledges support of the NSERC grant RGPIN-2022-02951.

\section{Preliminaries}

A \defit{monoid} is a non-empty set $H$ together with an associative operation $\cdot\colon H \times H \to H$ and a neutral element $1$.
At this point we make no assumption on the cancellativity of $H$.
By $H^\times$ we denote the group of units of $H$.
A nonunit $u \in H$ is an \defit{atom} (or an \defit{irreducible element}) if $u=ab$ with $a$,~$b \in H$ implies $a \in H^\times$ or $b \in H^\times$.
The monoid $H$ is \defit{atomic} if every nonunit of $H$ can be represented as a product of atoms.
It is well-known that every cancellative monoid satisfying the ACC on principal left ideals and the ACC on principal right ideals is
atomic \cite[Lemma 3.4]{smertnig16} and  \cite[Lemma 3.1]{smertnig13}. 

The \defit{length set} of a nonunit $a \in H$ is
\[
  \sL(a) = \{\, k \in \bZ_{\ge 0} \mid a=u_1 \cdots u_k \text{ with atoms } u_1, \ldots\,, u_k \,\}.
\]
Trivially $\sL(a) + \sL(b) \subseteq \sL(ab)$, and in particular $\sup \sL(a) + \sup \sL(b) \le \sup \sL(ab)$ if these sets are non-empty.
For $a \in H^\times$ we set $\sL(a) = \{0\}$.
\footnote{If $H$ is Dedekind-finite, that is, every left or right divisor of a unit is again a unit, this is a reasonable and convenient definition and preserves the inequality $\sup \sL(a) + \sup \sL(b) \le \sup \sL(ab)$ also when $a$ or $b$ is a unit. If $H$ is not Dedekind-finite, then this definition is somewhat dangerous, as a unit may then possibly be also represented as a non-trivial product of atoms.
Since we will soon restrict to (unit-)cancellative monoids, which are always Dedekind-finite, this will not pose any problem.}

If $a \in H$ and $\card{\sL(a)} \ge 2$ then there exist factorizations $a=u_1\cdots u_k=v_1\cdots v_l$ with atoms $u_i$,~$v_j \in H$ and $l > k$.
Then $a^n = (u_1\cdots u_k)^{m}(v_1\cdots v_l)^{n-m}$ for all $0 \le m \le n$, and therefore 
$\{\, m k + (n-m)l : 0 \le m \le n \,\} \subseteq \sL(a^n)$. 
Thus $\card{\sL(a^n)} \ge n+1$, and so we cannot expect a uniform bound on the length sets of $H$ (unless $\card{\sL(a)} \le 1$ for all $a \in H$).
We can however hope for the following basic finiteness property.

A monoid $H$ has \defit{bounded factorizations}, or in short, is a \defit{BF-monoid}, if $H$ is atomic and $\sL(a)$ is finite for all $a \in H$.
It is \defit{half-factorial} if $\card{\sL(a)}=1$ for all $a \in H$.

\begin{definition}
  Let $H$ be a monoid and let $\lambda \colon H \to \bZ_{\ge 0}$ be a function.
  \begin{enumerate}
  \item $\lambda$ is a \defit{right length function} if $\lambda(a) > \lambda(b)$ whenever $a=bc$ with $b$,~$c \in H$ and $c$ is a nonunit.
  \item $\lambda$ is a \defit{length function} if $\lambda(a) > \lambda(b)$ whenever $a=dbc$ with $c$,~$d \in H$ and at least one of $c$,~$d$ is a nonunit.
  \item $\lambda$ is a \defit{superadditive length function} if
    \begin{propenumerate}
    \item $\lambda(ab) \ge \lambda(a) + \lambda(b)$ for all $a$,~$b \in H$; and
    \item$\lambda(a) = 0$ implies $a \in H^\times$.
    \end{propenumerate}
  \end{enumerate}
\end{definition}

Every superadditive length function is a length function, and every length function is a right length function.
If $H$ is a commutative monoid, then $\lambda \colon H \to \bN_0$ is a length function if and only if it is a right length function.
It is classical that a commutative cancellative monoid is a BF-monoid if and only if it has a length function \cite[Proposition 1.3.2]{GHK06}.
This equivalence is extended to a noncommutative and possibly non-cancellative setting by Fan and Tringali in \cite[Corollary 2.29]{fan-tringali18} using the notion of a \emph{length function} as defined above.
For our purposes it will usually be more convenient to work with \emph{right length functions}. Our first goal, in \cref{t:bf}, is to show that the existence of \emph{any} of these types of length function is equivalent to $H$ having bounded factorizations.

A monoid $H$ is \defit{unit-cancellative} if $a=au$ or $a=ua$ with $a$,~$u \in H$ implies $u \in H^\times$.
Every cancellative monoid is unit-cancellative.
In a unit-cancellative monoid, every left [right] invertible element is invertible: if $uv=1$ then $uvu=u$ and hence $vu \in H^\times$. Hence $u$ also has a left inverse.
In particular, any right or left divisor of a unit is itself a unit.

\begin{lemma} \label{l:length-basic}
  Let $H$ be a monoid with right length function $\lambda$.
  \begin{enumerate}
  \item \label{length-basic:prod} If $a_1$, $\ldots\,$,~$a_k \in H$ with $k \ge 0$ are nonunits, then $\lambda(a_1 \cdots a_k) \ge k$.
    In particular, if $a \in H$ is a nonunit, then $\lambda(a) > 0$.
  \item \label{length-basic:cancellative} $H$ is unit-cancellative.
  \end{enumerate}
\end{lemma}

\begin{proof}
  \ref*{length-basic:prod}
  By induction on $k$.
  If $k=0$ then $a_1\cdots a_k=1$ (the empty product) is a unit, and the claims hold trivially.
  Suppose $k \ge 1$.
  Then $\lambda(a_1\cdots a_k) > \lambda(a_1 \cdots a_{k-1}) \ge k-1$.

  \ref*{length-basic:cancellative}
  Suppose $a = ab$ with $b$ a nonunit. Then $\lambda(a) > \lambda(a)$, a contradiction.
  Suppose now $a=ba$ with $b$ a nonunit.
  Then $a$ is also a nonunit and $a=b^ka$ for all $k \ge 0$.
  It follows that $\lambda(a) > \lambda(b^k) \ge k$ for all $k \ge 0$, a contradiction.
\end{proof}

The following characterization extends \cite[Corollary 2.29]{fan-tringali18} by superadditive length functions and right length functions.
We give a full proof for the convenience of the reader.

\begin{theorem} \label{t:bf}
  Let $H$ be a monoid.
  The following statements are equivalent.
  \begin{equivenumerate}
  \item\label{bf:al} $H$ has a superadditive length function.
  \item\label{bf:l} $H$ has a length function.
  \item\label{bf:rl} $H$ has a right length function.
  \item\label{bf:bf} $H$ is a BF-monoid.
  \item\label{bf:intersection} One has $\bigcap_{n \ge 0} (H\setminus H^\times)^n = \emptyset$.
  \end{equivenumerate}
  If these conditions are satisfied, then $H$ is unit-cancellative.
\end{theorem}

\begin{proof}
  \ref{bf:al}${}\Rightarrow{}$\ref{bf:l}${}\Rightarrow{}$\ref{bf:rl} is clear.

  \ref{bf:rl}${}\Rightarrow{}$\ref{bf:bf}
  Let $\lambda \colon H \to \bZ_{\ge 0}$ be a right length function.
  \Cref{l:length-basic} implies $\max \sL(a) \le \lambda(a)$ for all $a \in H$.
  We must still show that $H$ is atomic.
  Let $a \in H \setminus H^\times$.
  We proceed by induction on $\lambda(a)$.
  If $\lambda(a) = 0$ then $a$ is a unit by \ref{length-basic:prod} of \cref{l:length-basic} and there is nothing to show.
  Suppose $\lambda(a) > 0$.
  If $a$ is an atom, we are done.
  If $a$ is not an atom, then $a = b_1 c_1$ with nonunits $b_1$, $c_1 \in H$ and $\lambda(b_1) < \lambda(a)$.
  By induction hypothesis $b_1$ is a product of atoms.
  If $c_1$ is an atom we are done.
  Otherwise $c_1 = b_2 c_2$ with nonunits $b_2$,~$c_2 \in H$.
  Since $\lambda(b_1b_2) < \lambda(a)$, again $b_1b_2$ is a product of atoms.
  Continuing this process, we find
  \begin{equation} \label{eq:bf-atomic}
    a = b_1 \cdots b_k c_k
  \end{equation}
  with $b_1$, $\ldots\,$,~$b_k$, $c_k \in H\setminus H^\times$.
  Since $\lambda(a) > \lambda(b_1\cdots b_k) \ge k$, this process must terminate at some point, which means that eventually $c_k$ must be an atom.
  But then \cref{eq:bf-atomic} can be refined into a representation of $a$ as a product of atoms, as $b_1\cdots b_k$ is a product of atoms by induction hypothesis.

  \ref{bf:bf}${}\Rightarrow{}$\ref{bf:intersection}
  By contradiction.
  Let $a \in \bigcap_{n \ge 0} (H\setminus H^\times)^n$.
  Thus, for every $n \ge 0$ there exist $a_1$,~$\ldots\,$,~$a_n \in H \setminus H^\times$ such that $a=a_1\cdots a_n$.
  Since $H$ is atomic, each $a_i$ can be expressed as a product of atoms and therefore $\sup \sL(a) \ge n$.

  \ref{bf:intersection}${}\Rightarrow{}$\ref{bf:al}
  Due to the stated condition we may define $\lambda \colon H \to \bZ_{\ge 0}$ by $\lambda(a) = \max\{\, n \ge 0 : a \in (H \setminus H^\times)^n \,\}$.
  Then $\lambda$ is a superadditive length function.

  \smallskip
  That $H$ is unit-cancellative follows from \ref{bf:rl} together with \cref{l:length-basic}.
\end{proof}

Clearly \ref{bf:rl} may equivalently be replaced by a left length function.

Let $(H_i)_{i \in I}$ be a family of monoids.
The restricted product $\prod_{i \in I}' H_i$ is the submonoid of $\prod_{i\in I} H_i$ consisting of all $(\alpha_i)_{i \in I} \in \prod_{i\in I} H_i$ satisfying $\alpha_i \in H_i^\times$ for all but finitely many $i \in I$.
For a monoid $H$, the submonoid of cancellative elements is denoted by $H^\bullet$ and its center by $Z(H)$.
An intersection $H = \bigcap_{i \in I} H_i$ of monoids (in some common overmonoid $Q$) is  \defit{of finite type} if every $a \in H$ is a unit in all but finitely many of the $H_i$.

\begin{lemma} \label{l:bf-sub}
  \begin{enumerate}
  \item \label{bf-sub:hom} Let $H$ and $D$ be monoids and let $\varphi \colon H \to D$ be a monoid homomorphism with $\varphi^{-1}(D^\times) = H^\times$.
    If $D$ is a BF-monoid, then so is $H$.
  \item \label{bf-sub:sub} If $H \subseteq D$ are monoids such that $H \cap D^\times = H^\times$ and $D$ is a BF-monoid, then $H$ is a BF-monoid.
  \item \label{bf-sub:canc} If $H$ is a BF-monoid, then so are  $H^\bullet$ and $Z(H)$.
  \item \label{bf-sub:restprod} Restricted products of BF-monoids are BF-monoids.
  \item \label{bf-sub:intersection} Let $H \subseteq Q$ be monoids, and let $H = \bigcap_{i \in I} H_i$ be an intersection of finite type with overmonoids $H \subseteq H_i \subseteq Q$.
    If each $H_i$ is a BF-monoid, then so is $H$.
  \end{enumerate}
\end{lemma}

\begin{proof}
  \ref{bf-sub:hom}
  Let $\lambda \colon D \to \bZ_{\ge 0}$ be a right length function.
  Let $a$, $b$,~$c \in H$ such that $a=bc$ and $c \not \in H^\times$.
  Then $\varphi(a) = \varphi(b)\varphi(c)$.
  By assumption $\varphi(c) \not \in D^\times$, and hence $\lambda(\varphi(a)) > \lambda(\varphi(b))$. 
  Thus $\lambda \circ \varphi$ is a right length function on $H$, and $H$ is a BF-monoid.

  \ref{bf-sub:sub},~\ref{bf-sub:canc}
  Apply \ref{bf-sub:hom} to the inclusions $H \hookrightarrow D$,  $H^\bullet \hookrightarrow H$, respectively, $Z(H) \hookrightarrow H$.

  \ref{bf-sub:restprod}
  Let $H = \prod_{i \in I}' H_i$ with each $H_i$ a BF-monoid.
  Let $a=(\alpha_i)_{i\in I} \in H$ be a nonunit and let $I' = \{\, i \in I \mid \alpha_i \not \in H_i^\times \,\}$.
  Then $I'$ is finite.
  Let $a=a_1 \cdots a_k \in H$ with $a_j \in H$ nonunits and $a_j = (\alpha_{j,i})_{i \in I}$.
  For each $j \in [1,k]$ there exists an $i \in I$ with $\alpha_{j,i} \not \in H_i^\times$.
  Because each $H_i$ is unit-cancellative, even $i \in I'$.
  We conclude $k \le \sum_{i \in I'} \max \sL_{H_i}(\alpha_i)$.
  
  \ref{bf-sub:intersection}
  Let $a \in H$. If $a \in H_i^\times$ for all $i \in I$, then also $a^{-1} \in H$ and hence $a \in H^\times$.
  Thus the embedding $H \to \prod_{i\in I}' H_i$ satisfies the property required in \ref{bf-sub:sub}.
\end{proof}

\subsection*{The submonoid of cancellative elements and principal one-sided ideals}

As we have seen any BF-monoid is necessarily unit-cancellative.
We therefore restrict our attention to the submonoid of cancellative elements $H^\bullet$ of $H$.
Here an issue appears that needs some discussion: the factorizations of an element $a \in H^\bullet$ considered within $H^\bullet$ may differ from the factorizations of the same element as considered in $H$, because not every divisor of $a$ needs to be cancellative.
However, under reasonable conditions this is the case.

Let $S \subseteq H$ be a submonoid.
The submonoid $S$ is \defit{right saturated} in $H$ if, for all $a$,~$b \in S$ and $c \in H$ with $a=bc$ it follows that $c \in S$.
It is \defit{divisor-closed} if, for all $a \in S$ and $b$,~$c \in H$ with $a=bc$ it follows that $b$,~$c \in S$.
The submonoid $S \subseteq H$ is a \defit{right Ore set} if $aS \cap bH \ne \emptyset$ for all $a \in H$ and $b \in S$.

\begin{lemma} \label{l:cancellative}
  \begin{enumerate}
  \item \label{cancellative:saturated} Let $H$ be a monoid.
    If $H$ is cancellative or $H^\bullet$ is a right Ore set, then $H^\bullet$ is right saturated in $H$.
  \item \label{cancellative:divclosed}
    If $R$ is a domain or a prime Goldie ring, then $R^\bullet$ is divisor closed in $R$.
  \end{enumerate}
\end{lemma}

\begin{proof}
  \ref{cancellative:saturated}
  If $H$ is cancellative, then $H=H^\bullet$ and the claim holds trivially.
  
  Suppose that $H^\bullet$ is a right Ore set in $H$.
  Let $a$,~$b \in H^\bullet$ and let $a=bc$ with $c \in H$. If $x$,~$y \in H$ such that $cx=cy$, then $ax=bcx=bcy=ay$, and hence $x=y$.
  Let $x$,~$y \in H$ with $xc=yc$.
  There exist $b' \in H$ and $a' \in H^\bullet$ such that $ab'=ba'$.
  Then $ba'=ab'=bcb'$. Since $b \in H^\bullet$, we get $a'=cb'$.
  Then $xa'=xcb'=ycb'=ya'$ and $a' \in H^\bullet$ imply $x=y$. Hence $c \in H^\bullet$.

  \ref{cancellative:divclosed}
  If $R$ is a domain, the claim again holds trivially, so suppose that $R$ is a prime Goldie ring.
  Then $a \in R^\bullet$ if and only if $aR$ is an essential right ideal of $R$, if and only if $Ra$ is an essential left ideal of $R$.
  If $a \in R^\bullet$ and $a=bc$ with $b$,~$c \in R$, then $aR \subseteq bR$ and $Ra \subseteq Rc$ imply $b$,~$c \in R^\bullet$.
\end{proof}

If $H^\bullet \subseteq H$ is right saturated, there is a natural relationship between factorizations of elements of $H^\bullet$ and chains of principal right ideals of $H$, generated by cancellative elements.
For $a \in H^\bullet$, let $[aH,H] = \{\, bH : b \in H^\bullet,\ aH \subseteq bH \subseteq H \,\}$.
If $a=u_1\cdots u_k$ with atoms $u_1$, $\ldots\,$,~$u_k \in H^\bullet$, then
\[
  aH \subsetneq u_1\cdots u_{k-1} H \subsetneq \cdots \subsetneq u_1u_2 H \subsetneq u_1H \subsetneq H
\]
is a finite maximal chain in the poset $[aH,H]$.
Conversely, every finite maximal chain in $[aH,H]$ gives rise to a factorization of $a$ into atoms of $H^\bullet$.
Two factorizations correspond to the same chain if and only they are equal up to the insertion of units (for a formal treatment see \cite[Sections 3.1--3.2]{smertnig16}).

From this point of view we see (again assuming that $H^\bullet \subseteq H$ is right saturated):
\begin{enumerate}
\item $H^\bullet$ is atomic if and only if $[aH,H]$ contains a finite maximal chain for all $a \in H^\bullet$.
\item $H^\bullet$ has bounded factorizations if and only if, for every $a \in H^\bullet$, there exists a bound $\lambda(a)$ on the length of finite maximal chains of $H$.
\item $H^\bullet$ is half-factorial if and only if every $[aH,H]$ has a finite maximal chain and all finite maximal chains have the same length.
\end{enumerate}
If $[aH,H]$ contains only a finite, but nonzero, number of finite maximal chains, then $H^\bullet$ has \defit{finite factorizations}; see \cite{bell-heinle-levandovskyy17} for a sufficient condition in the noncommutative setting.

By the Jordan-Hölder theorem for modular lattices, we obtain the following.

\begin{lemma}
  If $H^\bullet$ is atomic and $[aH,H]$ is a modular lattice for all $a \in H^\bullet$, then $H^\bullet$ is half-factorial.
\end{lemma}

\begin{proof}
  Atomicity guarantees the existence of at least one finite maximal chain of some length $n$.
  By Jordan-Hölder (or the Schreier refinement theorem) every maximal chain has the same length $n$.
\end{proof}

\begin{remark}
  The study of factorizations in the presence of non-cancellative elements causes additional issues, already in the setting of commutative rings and monoids \cite{anderson-valdesleon96,anderson-valdesleon97}.
  Recent instances of non-cancellative monoids where factorizations have been studied are \cite{fan-tringali18,brantner-geroldinger-reinhart20,cossu-tringali21,geroldinger-khadam,bienvenu-geroldinger22}.
  In a ring with zero-divisors the correspondence between factorizations, as products of atoms, and maximal chains of principal right ideals breaks down.
  Facchini and Fassina \cite{facchini-fassina18} pursue the interesting approach of taking the latter concept, i.e., a maximal chain of principal right ideals, as the definition of a factorization. 
\end{remark}

\subsection*{Rings}
As we have seen in \cref{t:bf}, for $(R\setminus\{0\},\cdot)$ to be a BF-monoid, it has to be unit-cancellative.
In practice, this condition is too restrictive for the classes of rings that we are interested in.

For this reason we restrict our attention to the submonoid of cancellative elements $R^\bullet$, and make the following definition.
We will mostly be interesting in prime rings and domains, where this set is ``large''.

\begin{definition}
  A ring $R$ has \defit{bounded factorizations (BF)}, or is a \defit{BF-ring}, if the monoid of non-zero-divisors $R^\bullet$ is a BF-monoid.
\end{definition}

\section{Basic sufficient conditions for BF}

In this section we give some basic conditions that guarantee that a ring has BF (in the noncommutative setting).

\subsection{Small Krull dimension}

\begin{proposition} \label{p:smalldim} \label{p:artinian}
  Let $R$ be a ring.
  Then each of the following conditions implies that $R$ is a BF-ring.
  \begin{enumerate}
  \item \label{smalldim:artinian} $R$ is right artinian.
  \item \label{smalldim:onedim} $R$ is a right noetherian prime ring with $\rKdim R =1$.
  \end{enumerate}
\end{proposition}

\begin{proof}
  \ref{smalldim:artinian}
  By the descending chain condition on principal right ideals, every cancellative element of $R$ is a right unit.
  Since $R$ is also right noetherian, every right unit is in fact a unit.
  Thus $R^\bullet= R^\times$ and $R$ is trivially a BF-ring.

  \ref{smalldim:onedim}
  Let $a \in R^\bullet$.
  Because $R^\bullet$ is right saturated in $R$, it suffices to bound the length of a maximal chain in $[aR,R]$.
  But since $R/aR$ has finite length by \cite[Lemma 6.3.9]{mcconnell-robson01}, the length of such a chain is bounded by the length of $R/aR$.
\end{proof}

In \ref{smalldim:onedim}, setting $\lambda(a)$ to be the length of $R/aR$, we see that $\lambda$ is a right length function. In particular, a prime right principal ideal ring is always a BF-ring. (More specifically, it is even similarity factorial, see \cite[Section 4.1]{smertnig16}.)

\subsection{Filtered and graded rings}

\begin{proposition} \label{p:filtration}
  Let $R$ be a filtered ring with a filtration $R_0 \subseteq R_1 \subseteq \cdots$ such that $R_0 \subseteq R^\times$ and $\gr R$ is a domain.
  Then $R$ is a BF-domain.
\end{proposition}

\begin{proof}
  For $a \in R\setminus\{0\}$ define $\lambda(a) = \min\{\, i \in \bZ_{\ge 0} : a \in R_i \,\}$.
  The fact that $\gr R$ is a domain is equivalent to:
  for all $a \in R_i\setminus R_{i-1}$ and $b \in R_{j}\setminus R_{j-1}$, 
   with $i$, $j \ge 0$, one has $ab \in R_{i+j} \setminus R_{i+j-1}$ (with $R_{-1}=\emptyset$).
  We conclude that $\lambda(ab) = \lambda(a) + \lambda(b)$ for all $a$,~$b \in R \setminus\{0\}$.
  Thus $\lambda$ is a superadditive length function and $R$ is a BF-domain.
\end{proof}

\begin{corollary}
  Universal enveloping algebras of Lie algebras over fields are BF-domains.
\end{corollary}

\begin{proof}
  By the Poincaré--Birkhoff--Witt theorem, every such algebra has an associated graded ring that is a commutative polynomial ring over the base field.
\end{proof}

Before discussing skew polynomial rings, we note the following easy lemma.

\begin{lemma} \label{l:noeth-units}
  Let $R$ be a ring.
  If there exists $a \in R^\bullet \setminus R^\times$, $x \in R^\bullet$, and $\varepsilon \in R^\times$ such that $ax=x\varepsilon$, then $R$ does not satisfy the ACC on principal left ideals.
  In particular, the ring $R$ is not left noetherian.
\end{lemma}

\begin{proof}
  Clearly $Rax \subseteq Rx$.
  If $Rax=Rx$, then $x=rax$ for some $r \in R$.
  Cancellativity of $x$ implies $1=ra$. Then $a=ara$ and cancellativity of $a$ implies $ar=1$, in contradiction to $a \not \in R^\times$.
  Thus $Rx\varepsilon =Rax \subsetneq Rx$. 
  Hence there is an infinite ascending chain
  \[
    Rx \subsetneq Rx\varepsilon^{-1} \subsetneq Rx\varepsilon^{-2} \subsetneq \cdots \subsetneq R. \qedhere
  \]
\end{proof}

\begin{proposition} \label{p:skew}
  Let $S$ be a BF-domain. Then each of the following rings is a BF-domain.
  \begin{enumerate}
  \item \label{skew:poly} The skew polynomial ring $S[x;\sigma,\delta]$ where $\sigma$ is an injective endomorphism of $S$ such that $\sigma(a) \in S^\times$ implies $a \in S^\times$, and $\delta$ is a $\sigma$-derivation.
  \item \label{skew:laurent} The Laurent polynomial ring $S[x^{\pm 1}; \sigma]$ where $\sigma$ is an automorphism of $S$.
  \end{enumerate}
\end{proposition}

\begin{proof}
  \ref*{skew:poly}
  Since $S$ is a BF-domain, the function $\mu\colon S^\bullet \to \bZ_{\ge 0}$, $a \mapsto \max \sL(a)$ is a superadditive length function.
  For $0 \ne f = \sum_{i=0}^n x^i a_i$ with $a_i \in S$ and $a_n \ne 0$, we define $\lambda(f) = \deg(f) + \mu(a_n)$.
  Suppose that $f = gh$ with $h \not \in R^\times$.
  Let $g = \sum_{i=0}^k x^i b_i$ with $b_i \in S$ and $b_k \ne 0$; and let
  $h = \sum_{i=0}^l x^i c_i$ with $c_i \in S$ and $c_l \ne 0$.

  Then $k+l = n$ and $a_n = \sigma^l(b_k) c_l$.

  We have $k \le n$ and, since $\mu$ is superadditive, $\mu(\sigma^l(b_k)) \le \mu(a_n)$.
  Since $h$ is a nonunit, either $k < n$ or $c_l$ is a nonunit.
  In the second case, $\mu(\sigma^l(b_k)) < \mu(a_n)$.
  By our assumption on $R$, nonunits are mapped to nonunits under $\sigma$, and hence $\mu(b_k) \le \mu(\sigma^l(b_k))$.

  Thus
  \[
    \lambda(g) = k + \mu(b_k) \le k + \mu(\sigma^l(b_k)) < n + \mu(a_n) = \lambda(f).
  \]
  Hence $\lambda$ is a right length function.

  \ref*{skew:laurent}
  Again, let $\mu \colon S^\bullet \to \bZ_{\ge 0},$ $ a \mapsto \max \sL(a)$ be a superadditive length function for $S^\bullet$.
  If $0 \ne f = \sum_{i=m}^n x^i a_i \in S[x^{\pm 1}; \sigma]$ with $m \le n$ and $a_n$,~$a_m \ne 0$, we define $\omega(f) = n - m$ and $f_- = a_m$.
  We show that $\lambda(f) = \omega(f) + \mu(f_-)$ is a right length function.

  Suppose that $f=gh$ with $h \not \in R^\times$.
  Since $f_- = \sigma^l(g_-) h_-$ for some $l \in \bZ$, we find $\mu(f_-) \ge \mu(\sigma^l(g_-)) + \mu(h_-) = \mu(g_-) + \mu(h_-)$.

  Now either $\mu(h_-) > 0$ or $\omega(h) > 0$.
  In the second case, clearly $\omega(f) > \omega(g)$.
  Hence, in either case, $\lambda(f) > \lambda(g)$.
\end{proof}

By \cref{l:noeth-units} if the skew polynomial ring $R=S[x;\sigma,\delta]$ is noetherian (which is really the case we are interested in), the additional condition on $\sigma$ is automatically satisfied.

We briefly summarize how various algebraic properties pass between skew [Laurent] polynomial rings and their base rings.

\begin{lemma} \label{l:sp}
  Let $S$ be a ring.
  \begin{enumerate}
  \item Let $R=S[x;\sigma,\delta]$ with $\sigma$ an endomorphism of $S$ and $\delta$ a $\sigma$-derivation.
    \begin{enumerate} \label{sp:sp}
    \item \label{sp:noeth-down} If $R$ is right noetherian, then $S$ is right noetherian.
    \item \label{sp:domain-down} If $R$ is a domain, then $S$ is a domain and $\sigma$ is injective.
    \item If $\sigma$ is injective and $S$ is a domain, then $R$ is a domain.
    \item If $\sigma$ is an automorphism and $S$ is right \textup{[}left\textup{]} noetherian, then $R$ is right \textup{[}left\textup{]} noetherian.
    \item \label{sp:domain-downup}{ If $R$ is noetherian domain, then $\sigma (a) \in R^\times $ implies that $a \in R ^\times $.  }
    \end{enumerate}
  \item Let $R=S[x^{\pm 1};\sigma]$ with $\sigma$ an automorphism of $S$.
    \begin{enumerate}
    \item $R$ is right \textup{[}left\textup{]} noetherian if and only if $S$ is right \textup{[}left\textup{]} noetherian.
    \item $R$ is domain if and only if $S$ is a domain.
    \end{enumerate}
  \end{enumerate}
\end{lemma}

\begin{proof}
  Most of these are obvious or well known; we just { discuss \ref{sp:noeth-down},  \ref{sp:domain-down}  and  \ref{sp:domain-downup} } of \ref{sp:sp}.

  Since $R$ is a free right $S$-module, it is faithfully flat.
  Thus the lattice of right ideals of $S$ embeds into the lattice of right ideals of $R$.
  Hence, if $R$ is right noetherian, so is $S$.

  Suppose $R$ is a domain.
  Clearly $S$ is a domain.
  Assume there exists $0 \ne a \in S$ such that $\sigma(a) = 0$.
  Since $S$ is a domain, $0 \ne a^2$.
  Since $\delta(a^2) = a\delta(a)$, we have $a^2 x = a \delta(a)$, and hence $a(ax - \delta(a)) = 0$.
  Thus $a$ is a zero-divisor. \\
  { For (v), note that by \ref{sp:domain-down} $\sigma$ is injective and 
  so if $\sigma (a) \in R^\times $, then $a \in R^{\bullet}$.  Now since $R$ is noetherian, by 
 \cref {l:noeth-units},  $a \in R^\times$.  }
\end{proof}

{ 
\begin{corollary} \label{cor:iterated}
If $R$ is a noetherian  iterated skew \textup{[}Laurent\textup{]} polynomial domain over a commutative  or a BF ring $S$, then $R$ is BF.
\end{corollary}
\begin{proof}
By \cref{l:sp}\ref{sp:sp}, $S$ is noetherian. So in case $S$ is commutative it is BF, see  \cite[Corollary 1.3.5]{GHK06} or \cref{commmore}. Now \cref{l:sp}\ref{sp:sp}\ref{sp:domain-downup} and \cref{p:skew}\ref{skew:poly}
give us the result. 
\end{proof}

}

\subsection{Multiplicative ideal theory}
In the study of non-unique factorizations in the commutative setting, Krull monoids and domains take a central role, and they are BF, as they are $v$-noetherian (see \cite[Theorem 2.2.9]{GHK06} for the result, and Chapters 2.3 and 2.10 of the same monograph for additional context on Krull monoids and Krull domains).

There are several definitions of Krull orders in the noncommutative setting \cite{jespers-wauters84}, and they are generally equivalent for prime PI rings \cite[Theorem 1.4]{jespers-wauters84}.
The definitions by Chamarie \cite{chamarie81} and by Marubayashi \cite{marubayashi75,marubayashi76,marubayashi78} agree for bounded prime Goldie rings, and appear to be the most common ones (see also the survey \cite{akalan-marubayashi16} and Chapter 2.2 of the monograph \cite{marubayashi-vanoystaeyen12}).
For these Krull orders we have the following.

\begin{proposition} \label{p: bounded}
  Bounded Krull orders are BF-rings.
\end{proposition}

\begin{proof}
  By \cite[Corollary 5.30]{smertnig13} and the remark following it.

  Alternatively, a bounded Krull order $R$ is an intersection of finite type of local prime principal ideal rings \cite[Theorem 3.3]{akalan-marubayashi16} \footnote{Here, and in \cite{akalan-marubayashi16}, a ring $S$ is \defit{local} if $S$ modulo its Jacobson radical is simple artinian.}.
  Prime principal ideal rings have BF by \cref{p:smalldim}.
  The claim follows from \ref{bf-sub:intersection} of \cref{l:bf-sub}, applied to $R^\bullet$.
\end{proof}

\begin{remark}
  \begin{enumerate}
  \item In the setting of semigroup algebras, typically a prime Goldie ring is called a Krull order if it is a maximal order and satisfies the ascending chain condition on divisorial (two-sided) ideals (see \cite[p.\,56]{jespers-okninski07} or \cite[p.\,256]{okninski16}).
    As already mentioned, when one restricts to the setting of prime PI rings, this coincides with the other notions of Krull orders.
    In general, Krull orders in the sense of Chamarie are Krull orders in the sense of \cite{jespers-okninski07}.
    However, for the latter, it is still open whether every divisorial prime ideal is localizable or not, which poses an obstacle to the development of a smooth structure theory for such orders.
    We refer to the remark after \cite[Proposition 3.10]{akalan-marubayashi16} for more details and further references.

    \item
  In \cite{Geroldinger13} a noncommutative version of Krull monoids, originally introduced by Wauters, is studied.
  If $H$ is such a Krull monoid and every nonunit $a \in H$ is contained in a divisorial prime ideal, then $H$ is a BF-monoid \cite[Theorem 6.5]{Geroldinger13}.
  \end{enumerate}
\end{remark}

\subsection{Centers}

It is well-known that the center of a noetherian ring need not be noetherian.
Even noetherian prime PI rings need not have noetherian center.
The weaker $v$-noetherian property does pass to the center, and allows us to show that the center of a noetherian prime ring has BF, as follows.

Let $H$ be a monoid for which $H^\bullet$ is a right Ore set, and denote by $Q=Q(H)$ the right quotient monoid of $H$ by $H^\bullet$.
For $X \subseteq H$ define $\lc{H}{X} = \{\, q \in Q \mid qX \subseteq H \,\}$ and $\rc{H}{X} = \{\, q \in Q \mid Xq \subseteq H \,\}$.
A right ideal $I$ of $H$ is a \defit{right $H$-ideal} if $I \cap H^\bullet \ne \emptyset$.
We define $I_{H,v} \coloneqq I_v \coloneqq \rc{H}{\lc{H}{I}}$ and call $I$ \defit{divisorial} if $I=I_v$.
The monoid $H$ is \defit{right $v$-noetherian} if it satisfies the ascending chain condition on divisorial right $H$-ideals.
(These notions are developed in detail in the context of noncommutative monoids in \cite[Section 5]{smertnig13}, but go back to work of Asano and Murata \cite{asano-murata53}.)

The following is basically well-known; e.g. it is analogous to \cite[Proposition 5.1.10(a)]{mcconnell-robson01}.

\begin{lemma}\label{vnoeth}
  If $H$ is right $v$-noetherian, then $Z = Z(H) \cap H^\bullet$ is $v$-noetherian.
  In particular, $Z$ is a BF-monoid.
\end{lemma}

\begin{proof}
  The monoid $Z$ is commutative and cancellative.

  Let $I=I_{Z,v} \subseteq Z$ be a divisorial $Z$-ideal.
  We claim $(IH)_{H,v} \cap Z = I$.
  The inclusion ``$\supseteq$'' is trivial.
  To show ``$\subseteq$'' it suffices to show
  \[
    \rc{H}{\lc{H}{I}} \cap Z \subseteq \cc{Z}{\cc{Z}{I}}.
  \]
  where the colon ideals are taken in $Q(H)$ on the left and in $Q(Z)$ on the right.
  Let $x \in Z$ be such that $\lc{H}{I} x \subseteq H$.
  Then $\cc{Z}{I} x \subseteq H \cap Q(Z) \subseteq Z$.
  Thus $x \in \cc{Z}{\cc{Z}{I}}$, as claimed.

  Thus, since $H$ is right $v$-noetherian, so is $Z$.
  Hence $Z$ is a commutative cancellative $v$-noetherian monoid, and therefore has BF \cite[Theorem 2.2.9]{GHK06}.
\end{proof}

\begin{proposition} \label{p:center}
  Let $R$ be a right $v$-noetherian prime ring.
  Then $Z(R)$ is a BF-domain. In particular, if $R$ is a prime right noetherian ring, then $Z(R)$ is a BF-domain.
\end{proposition}

\begin{proof} The first claim follows from Lemma \ref{vnoeth}, and the second from the fact that right noetherian rings are right $v$-noetherian.
\end{proof}

\section{FBN rings with additional assumptions}

In this section we consider prime fully bounded noetherian (FBN) rings in which every nonzero ideal contains a nonzero central element.
In particular, this class includes all noetherian prime PI rings \cite[Theorem 13.6.4 and Corollary 13.6.6]{mcconnell-robson01}.
We give a second (different) proof for noetherian prime PI rings in \cref{subsec:pi}.
We recall that factor rings of FBN rings are FBN rings \cite[Exercise 9G]{goodearl-warfield04}.

We define a slight variation of the reduced rank.
Let $R$ be a semiprime right Goldie ring, and let $P_1$, $\ldots\,$, $P_n$ denote the pairwise distinct minimal prime ideals of $R$.
The semisimple right quotient ring $Q=Q(R)$ decomposes as $Q = \bigoplus_{i=1}^n Q_i$ with $Q_i=Qe_i$ and $e_1$, $\ldots\,$,~$e_n$ central idempotents of $Q$.
Each $R/P_i$ is a prime Goldie ring with $Q(R/P_i) \cong Q_i$ a simple artinian ring.

For a right $R$-module $M$, we define $\val_{P_i}(M) \in \bZ_{\ge 0} \cup \{\infty\}$ to be the length of the $Q_i$-module $M/M P_i \otimes_{R/P_i} Q_i$.
This is the same as the reduced rank of the $R/P_i$-module $M/M P_i$.
Alternatively, since $M/M P_i \cong M \otimes_R R/P_i$ and $R/P_i \otimes_{R/P_i} Q_i \cong Q_i \cong Q \otimes_Q Q_i$, we see that $\val_{P_i}(M)$ counts the multiplicity of the unique simple $Q_i$-module in the semisimple $Q$-module $M \otimes_R Q$.

\begin{lemma} \label{l:val-p}
  Let $R$ be a semiprime right Goldie ring and $P$ a minimal prime ideal of $R$.
  \begin{enumerate}
  \item\label{l:val-p:add} The map $\val_P\colon \operatorname{Mod-}R \to \bZ_{\ge 0} \cup \{\infty\}$ is additive on short exact sequences.
  \item\label{l:val-p:zero} $\val_P(M) = 0$ if and only if $M/MP$ is an $R/P$-torsion module.
  \end{enumerate}
\end{lemma}

\begin{proof}
  \ref*{l:val-p:add}
  Let
  $
    \begin{tikzcd}[column sep=small]
      0 \ar[r] & A \ar[r] & B \ar[r] & C \ar[r] & 0
    \end{tikzcd}
  $
  be a short exact sequence.
  Since ${}_R Q$ is flat, tensoring with it preserves the short exact sequence.

  \ref*{l:val-p:zero} Clear.
\end{proof}

\begin{lemma} \label{l:val-p-refinement}
  Let $R$ be a right noetherian ring with prime radical $N$, a minimal prime ideal $P$, and let $M$ be a right $R$-module.
  If $M=M_0 \supseteq M_1 \supseteq M_2 \supseteq \cdots \supseteq M_m=0$ and $M = M_0' \supseteq M_1' \supseteq M_2' \supseteq \cdots \supseteq M_n'=0$ are chains of submodules such that $M_iN \subseteq M_{i+1}$ and $M_i'N \subseteq M_{i+1}'$ for all $i$, then
  \[
    \sum_{i=1}^m \val_{P/N}(M_i/M_{i+1}) = \sum_{i=1}^n \val_{P/N}(M_i'/M_{i+1}').
  \]
\end{lemma}

\begin{proof}
  The additivity of $\val_{P/N}$ implies that the value of either sum is unchanged under refinement.
  The Schreier Refinement Theorem allows the chains to be refined to equivalent ones, which yields the desired conclusion.
  (This proof is the same as the one in \cite[Lemma 11.1]{goodearl-warfield04} for the reduced rank.)
\end{proof}

\begin{definition}
  Let $R$ be a right noetherian ring, $P$ a minimal prime ideal of $R$, and $N$ the prime radical of $R$.
  For a right module $M$, we define $\val_P(M)$ using the formula in \cref{l:val-p-refinement}.
  (Since $N$ is nilpotent, $M \supseteq MN \supseteq MN^2 \supseteq \cdots \supseteq 0$ shows the existence of a chain as required.)
\end{definition}

\begin{remark}
   The reduced rank of $M$ is $\sum_{P} \val_P(M)$ where the sum runs over all minimal prime ideals of $R$.

   In \cite{jategaonkar75}, Jategaonkar introduces a similar valuation for annihilators of the factors of a critical composition series of $M$.
   He shows that the valuation is additive for \emph{dominant} primes, these being the ones minimal among all the annihilator primes of the critical composition series of $M$.
   The minimal primes over $\ann(M)$ are precisely the dominant primes, so that we can recover Jategaonkar's valuations for dominant primes by considering $M$ as an $R/\ann(M)$ module.
 \end{remark}

 If $M$ is a right $R$-module with $A = \ann(M)$ and $P$ is a prime ideal minimal over $A$, then the expression $\val_{P/A}(M)$, interpreted in $R/A$ makes sense, since $P/A$ is minimal in $R/A$.
 The following lemma shows that we may replace $A$ by a smaller ideal $B \subseteq A$, while preserving $\val_{P/B}(M)=\val_{P/A}(M)$, as long as $P$ remains minimal over $B$.
 We will make use of this refinement later on.
 
 \begin{lemma} \label{l:varying-ideal}
   Let $R$ be a right noetherian ring and $M$ a right $R$-module.
   Let further $A = \ann(M)$ and let $B$ an ideal of $R$ and $P$ a prime ideal of $R$ such that $B \subseteq A \subseteq P$ and $P$ is minimal over $B$. \textup{(}Then $P$ is also minimal over $A$.\textup{)}
   Then
   \[
     \val_{P/A}(M) = \val_{P/B}(M),
   \]
   where the left side is computed in the ring $R/A$ and the right side in $R/B$.
 \end{lemma}

 \begin{proof}
   Let $N$ be the prime radical of $B$.
   Because $R/B$ is right noetherian, there exists $n \ge 1$ such that $N^n \subseteq B \subseteq A$.
   Consider the chain of $R$-modules
   \[
     M \supseteq M N \supseteq M N^2 \supseteq \cdots \supseteq MN^n = 0.
   \]
   Since $M$ is annihilated by $A$, the same is true for each factor $L_i \coloneqq MN^{i-1}/MN^{i}$ with $1 \le i \le n$.
   Then
   \[
     L_i \otimes_R R/P \cong L_i \otimes_{R/A} (R/A)/(P/A) \cong L_i \otimes_{R/B} (R/B)/(P/B),
   \]
   and of course $R/P \cong (R/A)/(P/A) \cong (R/B)/(P/B)$.
   Thus $\val_{P/B}(L_i)=\val_{P/A}(L_i)$ for all $i$.
   (These expressions make sense, because $P$ is minimal over both $A$ and $B$.)
   Finally  $\val_{P/A}(M) = \sum_{i=1}^n \val_{P/A}(L_i) = \sum_{i=1}^n \val_{P/B}(L_i) = \val_{P/B}(M)$.
 \end{proof}

 We will make use of the following variant of the Principal Ideal Theorem by Chatters, Goldie, Hajarnavis and Lenagan.
 They state the result for prime PI rings, but the same proof works for prime FBN rings as long as every nonzero ideal contains a nonzero central element.
 By $\cC(P)$ we mean the set of all elements $a \in R$ such that $a+P \in (R/P)^\bullet$.

\begin{proposition}[{\cite[Theorem 4.8]{chatters-goldie-hajarnavis-lenagan79}}] \label{p:principal-ideal-thm}
  Let $R$ be a prime FBN ring such that every nonzero ideal contains a nonzero central element and let $a \in R^\bullet$.
  If $P$ is a prime ideal minimal over $\ann(R/aR)$ and $a \not \in \cC(P)$, then $P$ has height at most $1$.
\end{proposition}

\begin{corollary} \label{c:pit}
  Let $R$ be a prime FBN ring and assume that every nonzero ideal contains a nonzero central element.
  For each $a \in R^\bullet \setminus R^\times$, there exists a height-one prime ideal $P$ with $\ann(R/aR) \subseteq P$.
\end{corollary}

\begin{proof}
  We follow the argument in \cite[Theorem 2.2]{chatters-gilchrist86}.

  Let $A = \ann(R/aR)$. Then $A$ is the maximal two-sided ideal contained in $aR$, and thus $0 \ne A \subsetneq R$, where $A \neq 0$ because $R$ is bounded and $aR$ is an essential right ideal by \cite[Proposition 2.3.4]{mcconnell-robson01}.
  Let $N$ be the prime radical of $A$.
  By \cite[Lemma 2.1]{chatters-gilchrist86}, $a \not \in \cC(N)$.
  Thus there exists a prime ideal $P$ minimal over $A$ such that $a \not \in \cC(P)$.
  By \cref{p:principal-ideal-thm}, the height of $P$ is $1$.
\end{proof}

\begin{lemma} \label{l:bounded-ann}
  Let $R$ be a right bounded semiprime right Goldie ring.
  If $M_1$, $\ldots\,$,~$M_n$ are finitely generated torsion modules, then there exists a nonzero ideal $I \subseteq R$ such that $M_i I = 0$ for all $i \in [1,n]$.
\end{lemma}

\begin{proof}
  For each $M_i$, let $m_{i,1}$, $\ldots\,$,~$m_{i,k}$ be generators.
  Choose $x_{i,j} \in R^\bullet$ with $m_{1,1} x_{1,1} =0$, $(m_{1,2}x_{1,1})x_{1,2} = 0$, $\ldots\,$, $(m_{2,1}x_{1,1}\cdots x_{1,k})x_{2,1} = 0$ and so on.

  Setting $x = x_{1,1}\cdots x_{1,k} \cdots x_{n,1} \cdots x_{n,k} \in R^\bullet$, we have $m_{i,j} x = 0$ for all $i \in [1,n]$ and $j \in [1,k]$.
  Since $R$ is semiprime right Goldie, $xR$ is an essential right ideal of $R$ by \cite[Proposition 2.3.4]{mcconnell-robson01}.
  Since $R$ is right bounded, there exists a right essential two-sided ideal $I \subseteq xR$ of $R$.
  If $m = m_{i,1} r_1 + \cdots + m_{i,k} r_k \in M_i$, then $m I \subseteq \sum_{j=1}^k m_{i,j} I \subseteq \sum_{j=1}^k m_{i,j} xR  = 0$.
  Thus $0 \ne I \subseteq \ann(M_i)$.
\end{proof}

\begin{lemma} \label{l:vp-pos}
  Let $R$ be an FBN ring.
  If $a \in R^\bullet \setminus R^\times$ and $P$ is a prime ideal minimal over $A=\ann(R/aR)$, then $\val_{P/A}(R/aR) > 0$.
\end{lemma}

\begin{proof}
  Let $N$ be the prime radical of $A$, and let $P_1=P$, $P_2$, $\ldots\,$,~$P_k$ denote the pairwise distinct minimal prime ideals over $A$.
  Set $Y = P_2 \cap \cdots \cap P_k \cap R$.

  Suppose $\val_{P/A}(R/aR) = 0$, and let $R/aR=M_0 \supseteq M_1 \supseteq \cdots \supseteq M_n=0$ be a sequence of submodules with $M_{i-1}N \subseteq M_i$.
  Then $\val_{P/A}(M_{i-1}/M_i)=0$ for all $i$ by Lemma \ref{l:val-p}\ref{l:val-p:add}, and hence $M_{i-1}/(M_{i-1}P+M_i)$ is an $R/P$-torsion module.
  By \cref{l:bounded-ann} there exists an ideal $X \supsetneq P$ annihilating all $M_{i-1}/(M_{i-1}P+M_i)$.
  Noting that $Y$ annihilates $(M_{i-1}P+M_i)/M_i$, we see that $XY$ annihilates $M_{i-1}/M_i$.
  Thus $(XY)^n \subseteq A \subseteq N$, and hence $X P_2 \cdots P_n \subseteq N \subseteq P_1$.
  Thus $X \subseteq P_1$, a contradiction.
\end{proof}

\begin{theorem} \label{thm: FBN}
  Let $R$ be a prime FBN ring and suppose that every nonzero ideal contains a nonzero central element.
  Then $R$ is a BF-ring.
\end{theorem}

\begin{proof}
  Let $\fX(R)$ denote the set of height-one prime ideals of $R$.
  For $a \in R^\bullet$ and $A = \ann(R/aR)$ define
  \[
    \lambda(a) = \sum_{\substack{P \in \fX(R) \\ A \subseteq P}} \val_{P/A}(R/aR).
  \]
  We claim that $\lambda$ is a superadditive length function.

  Suppose $a \in R^\bullet \setminus R^\times$.
  Then $A=\ann(R/aR) \ne 0$, because $aR$ is an essential right ideal of $R$, and $R$ is bounded.
  By \cref{c:pit} there exists a height-one prime ideal $P$ with $\ann(R/aR) \subseteq P$.
  Since $A \ne 0$, the ideal $P$ is minimal over $\ann(R/aR)$.
  \Cref{l:vp-pos} implies $\val_{P/A}(R/aR) > 0$, and hence $\lambda(a) > 0$.
  Thus $\lambda(a)=0$ implies $a \in R^\times$.

  Suppose that $a=bc$ with $b$,~$c \in R^\bullet$ and let $B=\ann(R/bR)$ and $C = \ann(R/cR)$.
  Since $aR = bcR$ and $bcR \subseteq bR \subseteq R$, the modules $R/bR$ and $R/cR\cong bR/bcR$ appear as factor, respectively, submodule of $R/aR$.
  Thus $A \subseteq B$ and $A \subseteq C$.
  Using the additivity of $\val_{P/A}$, we find
  \[
    \begin{split}
      \lambda(a) &= \sum_{\substack{P \in \fX(R) \\ A \subseteq P}} \val_{P/A}(R/aR) = \sum_{\substack{P \in \fX(R) \\ A \subseteq P}} \big(\val_{P/A}(R/bR) + \val_{P/A}(R/cR)\big) \\
      & \ge \sum_{\substack{P \in \fX(R) \\ B \subseteq P}} \val_{P/A}(R/bR) + \sum_{\substack{P \in \fX(R) \\ C \subseteq P}} \val_{P/A}(R/cR) \\
      &= \sum_{\substack{P \in \fX(R) \\ B \subseteq P}} \val_{P/B}(R/bR) + \sum_{\substack{P \in \fX(R) \\ C \subseteq P}} \val_{P/C}(R/cR) = \lambda(b) + \lambda(c).
    \end{split}
  \]
  The equalities $\val_{P/A}(R/bR)=\val_{P/B}(R/bR)$ for $B \subseteq P$, and the analogous equalities for $R/cR$, are justified by \cref{l:varying-ideal}.
\end{proof}

\begin{corollary} \label{c:prime-pi}
  Let $R$ be a noetherian prime PI ring.
  Then $R$ is a BF-ring.
\end{corollary}

\begin{proof}
  Every noetherian prime PI ring is an FBN ring \cite[Theorem 13.6.4]{mcconnell-robson01}.
  Moreover, the existence of a central polynomial can be used to show that every nonzero ideal contains a nonzero central element \cite[Corollary 13.6.6]{mcconnell-robson01}.
  Therefore  \cref{thm: FBN} implies the claim.
\end{proof}

\subsection{PI rings} \label{subsec:pi}

We now give a different proof of \cref{c:prime-pi}, that is more in the spirit of PI theory.
A ring extension $R \subseteq S$ is \defit{finite centralizing}, if the module $S_R$ has a finite set of generators, each of which centralizes $R$.
Finite centralizing (more generally, finite normalizing) extensions have the \defit{lying over} property: for every prime ideal $P$ of $R$, there exists a prime ideal $Q$ of $S$ with $P = Q \cap R$ \cite[Theorems 10.2.9 and 10.2.4]{mcconnell-robson01}.

\begin{proposition} \label{p:pi-ext-units}
  If $R \subseteq S$ is a finite centralizing extension of PI rings, then $S^\times \cap R = R^\times$.
\end{proposition}

\begin{proof}
  Clearly $R^\times \subseteq R \cap S^\times$, and it suffices to show the other inclusion.
  Assume that there exists an $a \in R \setminus R^\times$ such that $a \in S^\times$.
  Then $a$ is contained in a maximal right ideal $M \subseteq R$.
  Let $P = \ann(R/M)$.
  Since $P$ is primitive, it is prime.
  By the lying over property, there exists a prime ideal $Q$ of $S$ such that $Q \cap R = P$.
  Thus $R/P \hookrightarrow S/Q$.
  
  Since $a + P$ is contained in the maximal right ideal $M+P$ of $R/P$, it is a nonunit in $R/P$.
  Since $R/P$ is a primitive PI ring, Kaplansky's Theorem \cite[Theorem 13.3.8]{mcconnell-robson01} implies that $R/P$ is simple artinian.
  Thus $a+P$ is a zero-divisor, and since $R/P$ embeds into $S/Q$, also $a+Q$ is a zero-divisor in $S/Q$.
  It follows that $a+Q$ is a nonunit, and hence $a \not \in S^\times$.
\end{proof}

For the remainder of this section, let $R$ be a prime PI ring and $K$ the quotient field of its center $Z(R)$.
Then $R$ has a quotient ring $A$ that is a central simple algebra over $K$.
Each $a \in A$ induces a $K$-endomorphism $\mu_a$ of $A$, given by left multiplication, $x \mapsto ax$.
The \defit{characteristic polynomial} of $a$ is the characteristic polynomial of $\mu_a$.
The \defit{norm} of $a$ is $N_{A/K}(a) = \det(\mu_a)$.

Let $t(R)$ denote the subring of $K$ generated over $Z(R)$ by all coefficients of characteristic polynomials of elements of $R$.
The subring $T(R) = t(R) R$ of $A$ is the \defit{trace ring} of $R$ \cite[\S 13.9]{mcconnell-robson01}.

\begin{lemma}
  We have $t(R)=Z(R)$ if and only if $R$ is integral over $Z(R)$.
\end{lemma}

\begin{proof}
  An element $r \in R$ is integral over $Z(R)$ if and only if its characteristic polynomial has coefficients in $Z(R)$.
  From this the claim follows immediately.
\end{proof}

\begin{lemma} \label{l:norm}
  Suppose that $R$ is integral over $Z(R)$.
  \begin{enumerate}
  \item \label{norm:mul} $N_{A/K} \colon A \to K$ is a homomorphism of multiplicative monoids and restricts to a multiplicative homomorphism $N_{R/Z(R)} \colon R \to Z(R)$.
  \item \label{norm:zd} $a \in R$ is a zero-divisor if and only if $N_{A/K}(a) = 0$.
  \item \label{norm:unit} $a \in R^\times$ if and only if $N_{A/K}(a) \in Z(R)^\times$.
  \end{enumerate}
\end{lemma}

\begin{proof}
  \ref*{norm:mul}
  The norm is multiplicative because the determinant is, and $N_{A/K}(1) = \det(\id_A) = 1$.
  By assumption $N_{A/K}(a) \in Z(R)$ for all $a \in R$.

  \ref*{norm:zd}
  $a \in R$ is a zero-divisor if and only if it is a nonunit in $A$.
  However, this is clearly equivalent to $\mu_a$ not being an isomorphism, that is, to $N_{A/K}(a) = 0$.

  \ref*{norm:unit}
  Let $a \in R$.
  If $a \in R^\times$, then \ref*{norm:mul} implies $N_{R/Z(R)}(a) \in Z(R)^\times$.

  Suppose $N_{R/Z(R)}(a) \in Z(R)^\times$, and let $X^n + a_{n-1} X^{n-1} + \cdots + a_0$ with $a_i \in Z(R)$ denote the characteristic polynomial of $a$.
  By the Cayley-Hamilton Theorem, $\mu_a^n + a_{n-1} \mu_a^{n-1} + \cdots + a_0 = 0$ in $\End_K(A)$.
  Thus
  \[
    (\mu_a^{n-1} + a_{n-1} \mu_a^{n-2} \cdots + a_1) \mu_a = -a_0 \in Z(R)^\times.
  \]
  We conclude that $\mu_a$ is also invertible as an element of $\End_{Z(R)} R$. Since $R$ embeds in $\End_{Z(R)} R$, it follows that $a \in R^\times$.
\end{proof}

\begin{proof}[Second proof of \cref{c:prime-pi}]
  Suppose first that $R$ is integral over its center.
  Then $Z(R)$ is a BF-domain by \cref{p:center}, and \cref{l:norm} implies that $N_{A/K}\colon R^\bullet \to Z(R)^\bullet$ is a monoid homomorphism satisfying that $N_{A/K}(a) \in Z(R)^\times$ implies $a \in R^\times$.
  Thus $\max \sL(a) \le \max \sL(N_{A/K}(a))$.

  In the general case we now reduce the question to the trace ring $T(R)$.
  Since $R$ is noetherian, so is $T(R)$ \cite[Proposition 13.9.11]{mcconnell-robson01}.
  Since $R$ and $T(R)$ are equivalent orders in the quotient ring $A$ \cite[Corollary 13.9.7]{mcconnell-robson01}, the ring $T(R)$ is also a prime PI ring.
  Moreover, $T(R)$ is integral over $t(R)$.
  Since $t(R) \subseteq Z(T(R))$, the ring $T(R)$ is also integral over its center.
  Thus $T(R)$ is a BF-ring by the case we have already established.

  The right $R$-module $T(R)_R$ is finitely generated \cite[Proposition 13.9.11]{mcconnell-robson01}, so there exist $t_1$, $\ldots\,$, $t_n \in t(R)$ that generate $T(R)_R$.
  Thus $T(R)$ is a finite centralizing extension of $R$.
  \Cref{p:pi-ext-units} implies $R^\times = T(R)^\times \cap R$, and hence \ref{bf-sub:sub} of \cref{l:bf-sub} applied to $R^\bullet \hookrightarrow T(R)^\bullet$ implies that $R$ is a BF-ring. (Note that $R^\bullet = A^\times \cap R = T(R)^\bullet \cap R$.)
\end{proof}

\begin{example} \label{ex:affine-pi-not-bf}
  Let $K$ be a field, and suppose that $R$ is a prime PI ring that is an affine $K$-algebra, but not necessarily noetherian.
  Then $T(R)$ is a noetherian ring \cite[Proposition 13.9.11]{mcconnell-robson01}.
  However, in this setting the BF-property does not necessarily descend to $R$.
  Consider \cite[Example 13.9.9]{mcconnell-robson01}: Let
  \[
    S \coloneqq \bQ[x,y,y^{-1}] \qquad\text{and}\qquad R \coloneqq \begin{pmatrix} S & xS \\ S & \bQ[y] + xS \end{pmatrix} \subseteq M_2(S).
  \]
  Then $R$ is an affine $\bQ$-algebra \cite[\S 10.2]{mcconnell-robson01}, and
  \[
    T(R) = \begin{pmatrix} S & xS \\ S & S \end{pmatrix}
  \]
  is noetherian.
  The conclusion of \cref{p:pi-ext-units} does not hold for $R \subseteq T(R)$, for instance
  \[
    \begin{pmatrix} 1 & 0  \\ 0 & y \end{pmatrix} \in (T(R)^\times \cap R) \setminus R^\times.
  \]
  Moreover $R$ is not atomic (and therefore not a BF-ring).
  To see this, consider elements of the form
  \begin{equation} \label{e:apb:special-form}
    A = \begin{pmatrix}
      a & xb \\
      c & xd
    \end{pmatrix} \in R \quad\text{with $a$, $b$, $c$,~$d \in S$ and $\det(A) \ne 0$}.
  \end{equation}
  Then $\det(A) \not \in S^\times$ implies $A \not \in T(R)^\times $, and hence $A \not \in R^\times$.
  From $\det(A) \ne 0$ we get $A \in R^\bullet \subseteq M_2(S)^\bullet$.
  The obvious factorization
  \[
    A = \begin{pmatrix} 1 & 0 \\ 0 & y \end{pmatrix} \begin{pmatrix} a & xb \\ cy^{-1} & xdy^{-1} \end{pmatrix}
  \]
  shows that $A$ is not an atom.

  Suppose
  \[
    A =
    \underbrace{\begin{pmatrix} a_1 & xb_1 \\ c_1 & d_1 \end{pmatrix}}_{A_1} \underbrace{\begin{pmatrix} a_2 & xb_2 \\ c_2 & d_2 \end{pmatrix}}_{A_2},
  \]
  with $a_i$, $b_i$,~$c_i \in S$, and $d_i \in \bQ[y] + xS$ for $i \in \{1,2\}$.
  Then $d_1 \in xS$ or $d_2 \in xS$ as $xd=x c_1 b_2 + d_1d_2 \in xS$.
  Thus one of the $A_i$ is again of a form as in \cref{e:apb:special-form}.
  Repeating this argument recursively, we see that every representation of $A$ as a product of nonunits must contain some factor $B$ that is of the form as in \cref{e:apb:special-form}.
  However, then $B$ is not an atom, and therefore the ring $R$ is not atomic.
\end{example}

\section{Algebras of quadratic growth} \label{sec:quadratic}

Throughout this section, let $R$ be an affine prime algebra over a field $K$  and let $\GKdim(R)$ denote the Gelfand-Kirillov-dimension of $R$.
Then it is well-known that $\GKdim(R) \in \{0,1\} \cup \bR_{\ge 2}$ (see \cite{GW}).
If $\GKdim(R)=0$, then $R$ is finite-dimensional, in particular artinian, and therefore has BF.
If $\GKdim(R) = 1$, then a theorem of Small and Warfield shows that $R$ is prime noetherian PI \cite{small-warfield84,small-stafford-warfield85}, and therefore $R$ has BF by \cref{c:prime-pi}.
Thus $\GKdim(R)=2$ is the smallest remaining case.
In the following we deal with the case of $R$ having quadratic growth.
If $R$ has quadratic growth, then $\GKdim R = 2$, but the converse is not true in general.

Recall that a \defit{frame} is finite-dimensional vector subspace $V \subseteq R$ containing $1$ and a generating set of $R$.
The algebra $R$ has \defit{quadratic growth} if there exists a frame $V$ and constants $c_1$, $c_2 \in \bR_{>0}$ such that $c_1 n^2 \le \dim V^n \le c_2 n^2$ for all $n \ge 1$.

By \cite{bell09}, an affine simple Goldie $K$-algebra of quadratic growth is noetherian and has Krull dimension 1.
Thus \cref {p:smalldim} applies, and such algebras are BF-rings.
We now consider affine noetherian prime algebras of quadratic growth, that are not necessarily simple.
Since the PI case is covered by \cref{c:prime-pi}, we will eventually further restrict to non-PI algebras.

\begin{lemma} \label{l:intersection}
  Let $S$ be a noetherian prime ring, and let $a \in S^\bullet$ be a nonunit.
  Then
  \[
    S^\bullet \cap \bigcap_{k \ge 0} a^kS = \emptyset.
  \]
\end{lemma}

\begin{proof}
  Suppose not, and let $b \in \bigcap_{k \ge 0} a^kR$ with $b$ cancellative.
  Then, for every $k \ge 0$, there exists a $c_k \in R$ such that $b= a^k c_k$.
  Cancellativity of $a$ implies $c_k = a c_{k+1}$.
  Since $b$ is cancellative, so is each of the $c_k$.
  Thus $Sc_0 \subsetneq \cdots \subsetneq Sc_k \subsetneq S c_{k+1} \subsetneq \cdots$ is an infinite ascending chain of principal left ideals, in contradiction to $S$ being left noetherian.
\end{proof}

\begin{definition}
  An element $a \in R^\bullet$ is an \defit{almost unit} if $\dim_K R/aR < \infty$.
\end{definition}

If $a$, $b \in R^\bullet$ are almost units, then $R \supseteq aR \supseteq abR$ and $aR/abR \cong R/bR$ show $\dim R/abR = \dim R/aR \cdot \dim R/bR$.
In particular, the almost units form a submonoid of $R^\bullet$.

We will proceed to show that in the setting of interest to us, every almost unit is in fact a unit (\cref{l:no-almost-units} below).
To do so, we need to use a result from the theory of generalized polynomial identities (GPIs).
To avoid unnecessary technicalities (since we will assume that $R$ is noetherian later on), assume that $R$ is a Goldie ring, so that it has a classical quotient ring.
Let $Q=Q(R)$ be the quotient ring of $R$ and let $Z=Z(Q)$ be its center.
We consider elements of the free product (coproduct) $Q *_Z Z[T]$ of the $Z$-algebras $Q$ and $Z[T]$ as generalized polynomials.
In this ring, coefficients from $Q$ do not in general commute with the indeterminate $T$, but elements of $Z$ do commute with $T$.
More specifically, we will only deal with linear generalized polynomials, that is, elements of the form
\[
  \sum_{i=0}^k a_i T b_i \qquad\text{with $a_i$,~$b_i \in Q$.}
\]
The $Z$-vector space of linear generalized polynomials is isomorphic to $Q \otimes_Z Q$ via $\sum_{i=0}^k a_i T b_i \mapsto \sum_{i=0}^k a_i \otimes b_i$ (see \cite[Remark 6.1.1]{beidar-martindale-mikhalev96}).
A reference for GPIs is \cite{beidar-martindale-mikhalev96}.

The following result was proven by Martindale \cite[Theorem 2]{martindale69}.
Another reference is \cite[Corollary 6.1.3]{beidar-martindale-mikhalev96}, noting that the expression below corresponds to a nonzero generalized polynomial identity.

\begin{lemma} \label{l:no-linear-gpi}
  Let $S$ be a prime Goldie ring, $Z=Z(Q(S))$ its extended center, and let $a_1$, $\ldots\,$, $a_n$, $b_1$, $\ldots\,$, $b_n \in S$ be such that $a_1$, $\ldots\,$,~$a_n$ are $Z$-linearly independent and $b_1 \ne 0$.
  Then
  \[
    \sum_{i=1}^n a_i x b_i \ne 0
  \]
  for some $x \in S$.
\end{lemma}

For the remainder of the section, let $R$ be an affine noetherian prime ring of quadratic growth that is not PI.

\begin{lemma} \label{l:no-almost-units}
  Every almost unit in $R$ is a unit.
\end{lemma}

\begin{proof}
  Suppose that $a \in R^\bullet \setminus R^{\times}$ with $\dim_K R/aR < \infty$.
  Since $a$ is not a unit, it is not algebraic over $K$.
  By a theorem of Bell and Smoktunowicz \cite[Theorem 1.3]{bell-smoktunowicz10}, the element $a$ is also not algebraic over the extended center $Z=Z(Q)$, where $Q=Q(R)$ is the quotient ring of $R$.

  Let $V \subseteq R$ be a finite-dimensional $K$-vector space such that $R = V \oplus_K aR$.
  Since $aR \cong R$ as right $R$-modules, we have $R = \sum_{i=0}^{n-1} a^i V \oplus a^n R$ for all $n \ge 0$.
  Writing $v_1$, $\ldots\,$,~$v_m$ for a $K$-basis of $V$, for all $b \in R$ and $n \ge 0$ we may write
  \[
    b \in \sum_{i=1}^m B_{i,n}(a) v_i + a^nR
  \]
  with $B_{i,n} \in K[x]$ polynomials depending on $i$, $n$, and $b$.

  For $j \in [0,m]$ we iteratively construct generalized linear polynomials 
  $$f_j(T) \in Q *_Z Z[T],$$
   which are in fact of the form $f_j = \sum_{i=0}^k a^i T c_i$ with $c_i \in R$, as follows.
  Let $f_0 = T$.
  For $j \in [1,m]$, and with $f_{j-1}$ already constructed, consider the ascending chain of right ideals $\big( f_{j-1}(v_j)R + \sum_{i=1}^n [f_{j-1}(v_j), a^i] R \big)_{n \ge 0}$, where $[f_{j-1}(v_j),a^i]$ denotes the commutator.
  Since $R$ is right noetherian the chain stabilizes, and hence there exist $n_j \ge 0$ and $r_{j,i} \in R$ with $r_{j,n_j}=1$ such that
  \[
    f_{j-1}(v_j) r_{j,0} + \sum_{i=1}^{n_j} [f_{j-1}(v_j), a^i] r_{j,i} = 0.
  \]
  We set $f_j = f_{j-1} r_{j,0} + \sum_{i=1}^{n_j} [f_{j-1}, a^i] r_{j,i}$.

  By construction $f_j(v_i) = 0$ if $i \le j$, and therefore
  \[
    f_j(b) \in \sum_{i=j+1}^m B_{i,n}(a) f_j(v_i)  + a^n R.
  \]
  We conclude that $f_m(b) \in a^nR$ for all $b \in R$ and all $n \ge 0$.
  Expanding $f_m$, we see that it is of the form $\sum_{i=0}^k a^i T c_i$ with $c_i \in R$ and $c_k= r_{1,n_1} \cdots r_{m,n_m} = 1$.

  \Cref{l:intersection} implies that $I=\bigcap_{n \ge 0} a^nR$ is not essential as right ideal of $R$.
  Hence, with $Q$ denoting the quotient ring of $R$, we have $IQ \ne Q$ and so there exists $0 \ne z \in Q$ such that $zIQ=0$.
  Clearing denominators, we may assume $z \in R \setminus\{0\}$.
  Hence $zf_m(b) = 0$ for all $b \in R$. 
  This contradicts \cref{l:no-linear-gpi}.
\end{proof}

\begin{lemma} \label{l:quad-upper-bound}
  Let $V$ be a frame of $R$ and $k \ge 0$.
  Then there exists a constant $C \in \bR_{>0}$ such that for infinitely many $n \ge k$,
  \[
    \dim V^n / V^{n-k} \le C n.
  \]
\end{lemma}

\begin{proof}
  Let $C \in \bR_{>0}$.
  Suppose there exists $N \ge k$ such that for all $n \ge N$, we have $\dim V^n / V^{n-k} > C n$.
  Set $d(n) = \dim V^n$.
  Then
  \[
    \begin{split}
      d(N + lk) &= d(N) + \sum_{i=1}^l\big( d(N+ik) - d(N+(i-1)k) \big) \\
      & > d(N) + C \sum_{i=1}^l (N+ik) \\
      &=d(N) + CNl + Ck \sum_{i=1}^l i.
    \end{split}
  \]
  Since the last sum grows quadratically in $l$, and this is true for arbitrarily large $C$, this contradicts the quadratic growth hypothesis.
\end{proof}

\begin{lemma} \label{l:quad-upper-bound2}
  Let $a \in R^\bullet$.
  Then there exists a $C \in \bR_{>0}$ such that
  \[
    \dim ((V^n+aR)/aR) \le C n
  \]for infinitely many $n$.
\end{lemma}

\begin{proof}
  Let $k \ge 0$ be such that $a \in V^k$ and let $n \ge k$.
  There is a homomorphism of $K$-vector spaces $\varphi\colon V^n \to R/aR$ with $\im \varphi = (V^n + aR)/aR$.
  Since $a \in V^k$, we have $aV^{n-k} \subseteq V^n$.
  Hence $\varphi$ induces an surjective vector space homomorphism $V^n / a V^{n-k} \to \im \varphi$.
  By \cref{l:quad-upper-bound}, there exists a $C \in \bR_{>0}$ such that $\dim V^n / a V^{n-k} \le Cn$ for infinitely many $n$.
  (Note that $\dim aV^{n-k} = \dim V^{n-k}$ since $a$ is a non-zero-divisor.)
\end{proof}

\begin{lemma} \label{l:quad-lower-bound}
  If $a$, $b \in R^\bullet$ with $aR \subsetneq bR \subseteq R$, then there exists $c \in \bZ_{\ge 0}$ such that
  \[
    \dim \Big(\big((V^n + aR) \cap  bR\big)/aR\Big) \ge n - c
  \]
  for all sufficiently large $n$.
\end{lemma}

\begin{proof}
  We write $\overline{V^n}$ for $(V^n + aR)/aR$.
  Let $k \ge 0$ be such that $b \in V^k$, and let $\overline{W} = \overline{V^k} \cap bR/aR$.
  Since $b + aR$ generates $bR/aR$, it follows that
  \[
    bR/aR = \bigcup_{n\ge 1} \overline{W} V^n.
  \]

  Suppose that there exists an $N \ge 1$ such that $\overline{W} V^{N+1} = \overline{W} V^N$.
  Then $\overline{W} V^{N+n} = \overline{W} V^N$ for all $n \ge 0$.
  This implies that $bR/aR$ is finite-dimensional, and hence $b^{-1}a \in R$ is an almost unit but not a unit, in contradiction to \cref{l:no-almost-units}.
  Hence we must have $\overline{W} V^n \subsetneq \overline{W} V^{n+1}$ for all $n \ge 0$.
  Observe $\overline{W} V^n \subseteq \overline{V^{n+k}} \cap bR/aR$.
  Thus
  \[
    \dim(\overline{V^{k+n}} \cap bR/aR) \ge \dim(\overline{W} V^{n}) \ge \dim(\overline{W}) + n. \qedhere
  \]
\end{proof}

\begin{theorem} \label{thm: growth}
  Let $K$ be a field and let $R$ be an affine noetherian prime $K$-algebra of quadratic growth.
  Then $R$ is a BF-ring.
\end{theorem}

\begin{proof}
  We may without loss of generality assume that $R$ is not a PI ring, as we have already proved the result for noetherian prime PI rings.
  For $a \in R^\bullet$ we define
  \[
    \lambda(a) = \Big\lfloor \liminf_{n \to \infty} \frac{\dim((V^n + aR)/aR)}{n} \Big\rfloor
  \]
  and claim that $\lambda$ is a right length function.
  By \cref{l:quad-upper-bound2}, the limit inferior is finite.
  Now, if $aR \subsetneq bR$ with $b \in R^\bullet$, then, using \cref{l:quad-lower-bound}, there exists $c \in \bZ_{\ge 0}$ such that
  \[
    \begin{split}
      \frac{\dim((V^n + aR)/aR)}{n} &= \frac{\dim(((V^n+aR) \cap bR) / aR)}{n} + \frac{\dim ((V^n + bR)/bR) }{n} \\
      &\ge 1 - \frac{c}{n} + \frac{\dim ((V^n + bR)/bR) }{n}.
    \end{split}
  \]
  Thus $\lambda(a) \ge 1 + \lambda(b)$.
\end{proof}

\begin{example}
  The ring $R$ from Example~\ref{ex:affine-pi-not-bf} is $\bQ$-affine with frame $V$ generated by $e_{11}$, $x e_{11}$, $y e_{11}$, $y^{-1} e_{11}$, $e_{21}$, $e_{12}x$, $e_{22}$, $e_{22}y$ (see \cite[\S 13.10.2]{mcconnell-robson01}; we add $e_{11}$ to have $1 \in V$, as is our convention).
  Now $V^n$ contains the linearly independent set $\{\, x^{i}y^{m-i}e_{11} : 0 \le m \le n,\, 0 \le i \le m \,\}$.
  On the other hand $V^n$ is contained in $\{\, x^{i}y^{m-i}e_{kl} : 0 \le m \le n,\, 0 \le i \le m,\, k, l \in \{1,2\} \,\}$.
  Thus $R$ is a prime PI $\bQ$-affine algebra of quadratic growth that is not atomic, and in particular, not a BF-ring.
  We see that the condition that $R$ be noetherian in \cref{thm: growth} cannot simply be dropped.
\end{example}


\section{Sufficient homological conditions for BF}\label{hom}

This section contains two main results. When restricted to algebras over a field, the first of these results (Corollary \ref{AG}) is a special case of the second (Theorem \ref{gradejump}). Despite this, we have included a separate proof of Corollary \ref{AG} before proceeding to Theorem \ref{gradejump}, since it applies to rings rather than to algebras. The basic underlying idea is similar in both cases - namely, a function from (isomorphism classes of) modules to ordinals is defined, with properties which preclude the possibility of factorizations of unbounded length. But in order to achieve this the second result requires significantly more heavy-duty technology than the first.

In $\S$\ref{AusGor} we state a simple and very general lemma, Lemma \ref{easyandimportant} and then apply it to prove in Corollary \ref{AG} that every Auslander Gorenstein noetherian ring is a BF-ring. In $\S$\ref{dualizing} the concept of an Auslander dualizing complex is recalled from \cite{YZ}, and used to prove Theorem \ref{gradejump}, yielding for example property BF for all prime noetherian algebras which are factors of an algebra with such a dualizing complex. Consequences for noetherian algebras of finite Gelfand-Kirillov dimension are discussed in Corollary \ref{GKBF}. Finally, in $\S$\ref{commagain} we show that for a commutative noetherian ring $R$ there is a very natural (and non-homological) choice of a function $j$ satisfying the hypotheses of Lemma \ref{easyandimportant}, hence yielding another proof of property BF for commutative noetherian domains. However, as we show by an example due to Hochster and Heitmann suggested to us, this function does not satisfy Gabber's Maximality Principle, which is key to the proof of Theorem \ref{gradejump}.

\subsection{Auslander-Gorenstein rings}\label{AusGor}
\begin{notation} \label{notation} {\rm Suppose that $j$ is a map from $ \mathrm{Mod} _f (R)$, a representative set of isomorphism classes of finitely generated right
$R$-modules, to the set of ordinal numbers. Then:\\
(1) $M \in \mathrm{Mod}_f (R)$ is called \defit{$j$-pure} if $j (M) = j(N)$  for every nonzero submodule $N$ of $M$. \\
(2) $j$ is called \defit{exact}, if for every finitely generated module $M$, and $0 \subseteq N \subseteq M$, 
$j (M) = \mathrm{inf} \{j (N) , j(M/N)\}$.\\
(3) If $\alpha$ is an ordinal number, then $j$ is called \defit{finitely partitive on $\alpha$} if for every  $M \in \mathrm{Mod}_f (R)   $
 with $j(M) = \alpha$, there is a finite bound on the length of chains of submodules of $M$ of the form
$M_0 \subseteq M_1 \subseteq \cdots \subseteq M_n = M$ with $j(M_{i + 1}/M_{i}) = \alpha$ for every $0 \leq i < n $. \\
(4) $j$ is said to satisfy the \defit{torsion property on $R$}, if $j(R/xR) \geq j(R) + 1$, for every nonunit regular element $x$ of $R$.  }
\end{notation}

Given a map $j$ from $\mathrm{Mod} _f (R)$ to a set of ordinal numbers and an infinitely generated module $M$,   we can 
 define 
$j(M) $ to be $\mathrm{inf}\{\, j (N)\, : \, 0 \neq N \subseteq M, N $  finitely generated $\}$. If $j$ is exact, then this extension remains exact. That is,  if $M$ is a right module, not necessarily finitely generated, then again $j (M) = \mathrm{inf}\{j(N) , j (M/N)\}$ for every nonzero submodule $N$ of $M$. 
To prove this, let $ 0  \subseteq N \subseteq M$ be a submodule of $M$, then 
by definition of $j$, it is easy to see that $j(M) \leq j(N)$. Now let $j (M/ N) = j (k_1 R + k_2 R +  \cdots +  k_n  R + N /  N)$. Then 
$j (M/ N ) = j(k_1 R + k_2 R +  \cdots +  k_n  R / (N \cap k_1R  + k_2R  +  \cdots +  k_n R )) \geq  j(k_1 R + k_2 R +  \cdots +  k_n R  ) \geq j(M)$. 

There exists a finitely generated submodule $N'$ of $M$ such that  $j(M) = j(N')$. Therefore,  $j (N') = j (N\cap N') $ or 
 $j (N') = j (N'/ N'\cap N) $. In the first case, we have  $j (N) \leq j (N'\cap N) = j(M) $  and so $j(M) =j (N) $ and in the second case,   $j(M/ N) \leq j(N' + N / N) = j(M)$ and so the equality holds. \\

In the rest of $\S$\ref{hom}, given a ring $R$, $j$ will always denote a function from $\mathrm{Mod}_f (R)$ to a set of ordinal numbers, which is then extended as above to infinitely generated modules. Thus we can talk about not necessarily finitely generated $j$-pure modules in a similar fashion to  Notation \ref{notation}(1).  

A tight bound on the value of $j$ on cyclic modules of the form $R/xR$ will be important for us in the sequel, as the following lemma confirms.

\begin{lemma} \label{easyandimportant} Let $R$ be a noetherian ring for which there exists a map $j$ from $\mathrm{Mod} _f (R)$  to a set of ordinal numbers such that, for a fixed ordinal number $\alpha$, $j$ is finitely partitive on $\alpha$ and  $j(R/xR) =  \alpha $ for every $x \in R^{\bullet} \setminus R^{\times}$. Then $R$ is a BF-ring. 
\end{lemma}

\begin{proof} Let $R$ be as stated and let $x \in R^{\bullet} \setminus R^{\times}$, so $j (R/ xR) = \alpha$. 
 Suppose that $x = a_1 \cdots a_t$ with $a_i \in R^{\bullet} \setminus R^{\times}$ for all $i$. Then there is a chain of submodules  
$$ 0 = M_0 \subseteq M_1=  a_1\cdots a_{t-1}R/xR \subseteq \cdots  \subseteq M_{t-1}=  a_1R/ xR \subseteq  M_{t}= R/xR$$ 
of $R/xR$,  with $M_{i + 1} / M_{i} \cong R/ a_{t-i }R$ for all $i = 0, \ldots , t-1$. 
By hypothesis, $ j (M_{i + 1} / M_{i} ) = \alpha$ for all $i$. Hence, since $j$ is 
finitely partitive on $\alpha$, there is a bound on the length of such chains. Therefore, $R$ is a BF-ring.
\end{proof}

To give the first application of Lemma \ref{easyandimportant}, we need to recall some homological terminology. For more details, see \cite{Lev}, \cite{C} or \cite{Bj}, for example.

\begin{definition}\label{ausgor} Let $R$ be a noetherian ring.
\begin{enumerate}
\item[(1)] Let $M$ be a left or right $R$-module.  The \defit{(homological) grade} of $M$ is
$$ j(M) \; := \mathrm{inf}\{ i \,| \,\mathrm{Ext}^i_R (M,R) \neq 0 \} \; \in \; \mathbb{Z}_{\geq 0} \cup \{\omega\}.$$
In particular, $j(0) = \omega$, where $0$ denotes the zero $R$-module. 
\item[(2)] $R$ satisfies the \defit{Auslander condition} if, for every left or right $R$-module $M$ and every non-negative integer $i$, $j(N) \geq i$ for every submodule $N$ of $\mathrm{Ext}^i_R(M,R)$.
\item[(3)] $R$ is \defit{Auslander-Gorenstein} if
\begin{itemize}
\item[(i)] $R$ has finite (and equal) right and left injective dimensions;
\item[(ii)] $R$ satisfies the Auslander condition.
\end{itemize}
\item[(4)] $R$ is \defit{Auslander-regular} if it is Auslander-Gorenstein and has finite global (homological) dimension.
\end{enumerate}
\end{definition}

The fact that the (negative of the) grade yields an exact finitely partitive dimension function for Auslander-Gorenstein rings was first observed by Björk \cite[Theorem 1.17]{Bj}; a detailed account in this setting was given in \cite[Theorem 4.2, (4.6.5), (4.6.7)]{Lev}\footnote{Note that the definition of ``finitely partitive'' given in \cite[Proposition 4.5(iv)]{Lev} is weaker than Definition \ref{AusGor}; but Levasseur observes in \cite[(4.6.5)]{Lev} that the stronger conclusion is valid.}.

\begin{corollary}\label{AG} Every Auslander-Gorenstein ring is a BF-ring.
\end{corollary}
\begin{proof}
Let $R$ be an Auslander-Gorenstein ring  and let $x \in R^{\bullet} \setminus R^{\times}$. Then $\mathrm{Hom}_R (R/ xR, R) = 0 $. Moreover, since $ 0 \to xR \to R \to R /xR \to 0$ is a non-split exact sequence as $xR \cong R$, it follows that $\mathrm{Ext}_R ^1 (R/xR, R) \neq 0$, so that $j(R/xR) = 1$. However, as noted before the corollary, when $R$ is Auslander-Gorenstein the homological grade $j$ is finitely partitive for every ordinal, and in particular for $1$. Therefore $R$ is a BF-ring by Lemma \ref{easyandimportant}. 
\end{proof}

We list here some large and important classes of noetherian rings which are known to be Auslander-Gorenstein.

\begin{examples}\label{8.6}\begin{enumerate}
\item[(1)] At the time of writing, all known noetherian Hopf algebras are Auslander-Gorenstein. Whether this is in fact a theorem has been an open question for 25 years, see \cite[1.15]{BRGO97}. Many large classes of noetherian Hopf algebras are known to be Auslander-Gorenstein. These include:
\begin{itemize}
\item[$(\bullet)$] noetherian Hopf algebras satisfying a polynomial identity \cite[Theorem 0.1]{WZ};
\item[$(\bullet)$] group algebras of polycyclic-by-finite groups \cite[Theorem 6.7]{BRZH08};
\item[$(\bullet)$] enveloping algebras of finite dimensional Lie algebras \cite{E};
\item[$(\bullet)$] quantised enveloping algebras \cite[Proposition 2.2]{BRGO97};
\item[$(\bullet)$] connected Hopf algebras of finite Gelfand-Kirillov dimension \cite{Zh};
\item[$(\bullet)$] quantised coordinate rings of semisimple groups \cite[Theorem 0.1]{GZ}.
\end{itemize}
\item[(2)] A commutative noetherian ring is Auslander-Gorenstein if and only if it is Gorenstein; that is, if and only if it has finite injective dimension \cite{Ba}. 
\item[(3)] A ring with a locally finite $\mathbb{N}-$filtration whose associated graded ring is commutative Gorenstein is Auslander-Gorenstein. In particular, the Weyl algebras $A_n(k)$ over a field $k$ are Auslander-regular of global dimension $n$ in characteristic 0, and $2n$ in characteristic $p > 0$ \cite{E}.
\item[(4)] A local\footnote{By a \defit{local} ring we mean a ring whose factor by its Jacobson radical is simple artinian.} fully bounded noetherian ring of finite global dimension is Auslander-regular \cite {Te97}. 
\item[(5)]  Sklyanin algebras are Auslander-Gorenstein domains \cite{TaBe96}.
\end{enumerate}
\end{examples}

\bigskip


\subsection{Dualizing complexes}\label{dualizing} In this subsection we give two further applications of Lemma \ref{easyandimportant}. First, we strengthen the homological technology used to define the function $j$ in Corollary \ref{AG} so that a much larger class of noetherian rings are included than the Auslander-Gorenstein rings considered in $\S$\ref{AusGor}; the outcome is Theorem \ref{gradejump}. A cost of this (apart from the weight of equipment required) is that we have to work with algebras over a field, rather than rings.

In the rest of this section $k$ will denote a field, and all unadorned tensor products are assumed to be over $k$. The opposite ring of a $k$-algebra $R$ is
 denoted by $R^{\circ}$, and $R^e$ denotes the $k$-algebra $R\otimes R^{\circ}$, so that $\mathrm{Mod}(R^e)$ is the category of $R-R$-bimodules on which $k$ operates centrally. Given an algebra $R$, $\mathrm{D}(\mathrm{Mod}(R))$ (resp. $\mathrm{D}^b(\mathrm{Mod}(R))$ will denote the derived category (resp. the bounded derived category) of right $R$-modules; for details, see for example \cite {Ye20}.

The following definition, which is enough for our purposes here, is a special case of the more general version, for two possibly distinct algebras, given in \cite[Definition 1.1]{YZ}.

\begin{definition}\label{dual} Let $R$ be a noetherian $k$-algebra. A complex $\mathcal{R} \in \mathrm{D}^b(\mathrm{Mod}( R\otimes R^{\circ}))$ is called a
\defit{dualizing complex} over $R$ if it satisfies the three conditions below:
\begin{itemize}
\item[(i)] $ \mathcal{R}$ has finite injective dimension over $R$ and $R^{\circ}$;
\item[(ii)] $\mathcal{R}$ has finitely generated cohomology modules over $R$ and $R^{\circ}$;
\item[(iii)] The canonical morphisms $R^{\circ}\longrightarrow \mathrm{RHom}_{R}(\mathcal{R},\mathcal{R})$ and $R\longrightarrow \mathrm{RHom}_{R^{\circ}}(\mathcal{R},\mathcal{R})$ in $\mathrm{D}(\mathrm{Mod} (R^e))$ are isomorphisms.
\end{itemize}
\end{definition}

To say that the dualizing complex $ \mathcal{R}$ for the $k$-algebra $R$ satisfies the Auslander property simply means that the usual definition of the Auslander-Gorenstein condition - Definition \ref{ausgor}(3) - holds when the ring $R$ is replaced by $\mathcal{R}$. More precisely, we have the following definitions.

\begin{definition}\label{Ausdual}(Yekutieli, Zhang, \cite[Definitions 2.1,2.2]{YZ}) Let $\mathcal{R}$ be a dualizing complex for the noetherian $k$-algebra $R$.
\begin{enumerate}
\item[(1)] Let $M$ be a
finitely generated $R$-module. The \defit{grade} of $M$ with respect to $\mathcal{R}$ is
$$ j_{\mathcal{R};R}(M) \; := \; \mathrm{inf}\{j : \mathrm{Ext}^j_R(M,\mathcal{R}) \neq 0\} \in \mathbb{Z} \cup \{\infty\}.$$
A similar definition gives the grade $j_{\mathcal{R};R^{\circ}}(M')$ of an $R^{\circ}$-module $M'$.

\item[(2)] $\mathcal{R}$ is an \defit{Auslander dualizing complex} for $R$ if
\begin{itemize}
\item[(i)] for every finitely generated $R$-module $M$ and integer $q$, and for every nonzero finitely generated $R^{\circ}$-submodule $N$ of $\mathrm{Ext}^q_R(M,\mathcal{R})$,
$$ j_{\mathcal{R}; R^{\circ}}(N) \geq q;$$
\item[(ii)] the corresponding condition holds for finitely generated $R^{\circ}$-modules $M'$.
\end{itemize}
\end{enumerate}
\end{definition}

In the remainder of this section $R$ will always be used to denote an algebra over the field $k$ with Auslander dualizing complex $\mathcal{R}$. For convenience and to align our notation as far as possible with the first part of $\S$\ref{hom}, we will denote the maps $j_{\mathcal{R};R}$  and $ j_{\mathcal{R}; R^{\circ}}$ defined in Definition \ref{Ausdual} respectively by $j$ and $j^{\circ}$.
A nonzero right $R$-module $M$ will be called $j-$\defit{pure} if $j(N) = j(M)$ for all nonzero submodules $N$ of $M$; and similarly we may refer to $j^{\circ}-$pure left modules. Without loss of generality, 
because  $\mathcal{R}$ is of finite injective dimension, we can suppose that the grade is always non-negative.
 For consistency of notations throughout $\S$\ref{hom}, whenever $j(M) = \infty$ for an $R$-module $M$ we will write $j(M) = \omega$, and similarly for $j^{\circ}(M)$.
 
The link between the ideas of the first part of $\S$\ref{hom} and the present discussion is made clear by the following key result. 

\begin{theorem}\label{gradedim}({\rm Yekutieli-Zhang}, \cite[ Definition 2.4, Definition  2.9 and Theorem  2.10]{YZ}) Let $R$ be a noetherian $k$-algebra with an Auslander dualizing complex $\mathcal{R}$. Then $-j$ and $-j^{\circ}$ are finitely partitive exact dimension functions.
\end{theorem}

We will follow Yekutieli and Zhang \cite[Definition 2.9]{YZ} in calling $-j$ and $-j^{\circ}$ the \defit{canonical dimension functions} associated to $\mathcal{R}$.
It is important to note that the dimension function $-j$ can take its values over not necessarily finitely generated modules too as
  explained after Notation \ref{notation}. 

As is well known, the apparatus of critical modules and related technology can be invoked once one has available a finitely partitive exact dimension function. Summarising briefly for the present context, a nonzero finitely generated right $R$-module $M$ is $j$-\defit{critical} if every proper quotient of $M$ has $j$-dimension strictly bigger than
$j (M)$. Given a finitely generated right $R$-module $M$, a $j$-\defit{critical composition series} for $M$ is a finite chain
$$ 0 = M_0 \subseteq M_1 \subseteq  \cdots \subseteq M_t = M$$
of submodules $M_i$ of $M$, with the subfactors $M_i/M_{i-1}$ $j$-critical for $i = 1, \ldots , t$. We say that critical $R$-modules $C$ and $D$ are \defit{similar} if there is a nonzero $R$-module $X$ which embeds in each of them; equivalently, since critical modules are uniform, $C$ and $D$ are similar if and only if they have isomorphic injective hulls. One then has:

\begin{theorem}\label{ccs} Let $R$ be a noetherian $k$-algebra with an Auslander dualizing complex $\mathcal{R}$.
\begin{enumerate}
\item[(1)] (\cite[Corollary 2.17]{YZ}) Every finitely generated right $R$-module $M$ has a $j$-critical composition series.\\
\item[(2)] Suppose that $ 0 = M_0 \subseteq  M_1 \subseteq \cdots \subseteq M_t = M$ and $ 0 = N_0 \subseteq  N_1 \subseteq \cdots \subseteq N_s = M$ are two $j$-critical composition series for $M$. Then $s= t$ and there is a permutation $\sigma$ of $\{1, \ldots, t \}$ such that $M_i/M_{i-1}$ is similar to $N_{\sigma(i)}/N_{\sigma(i)-1}$ for all $i = 1, \ldots , t$.
\end{enumerate}
\end{theorem}
\begin{proof}(1) The argument given in \cite[Theorem 15.9]{goodearl-warfield04} for Krull dimension works also for $j$. One simply has to note that 
the finitely partitive property of $j$ guaranteed by Theorem \ref{gradedim}(1) ensures that every nonzero module $X$ contains a $j$-critical submodule $Y$. So, we fix $\alpha$ to be the least canonical dimension occurring amongst nonzero submodules of $M$, and let $M_1$ be maximal amongst $\alpha$-critical submodules of $M$. Then, repeat the procedure with $M/M_1$. 

(2) This can be proved by standard methods - for example, one may follow the argument for Krull-critical composition series given in \cite[Proposition 6.2.21]{mcconnell-robson01}.
\end{proof}

An easily overlooked but nevertheless fundamental point concerning the canonical dimension is that it always takes finite values on nonzero finitely generated modules:

\begin{proposition}\label{finite} Let $R$ and $\mathcal{R}$ be as in Theorem \ref{ccs}, and let $M$ be a nonzero finitely generated $R$-module. Then $j(M) \in \mathbb{Z}$.
\end{proposition}
\begin{proof} Suppose that $M$ is a nonzero finitely generated $R$-module with $j(M) = \omega$.
 Since $j$  is finitely partitive by Theorem \ref{gradedim}, we can replace $M$ by a submodule if necessary and so assume that $M$ 
 is $j$-critical with dimension $\omega$. But this contradicts the fact that every uniform injective $R$-module appears in any minimal injective resolution of $\mathcal{R}$, by \cite[Theorem 1.11(2)]{YZ} (whose proof is taken from \cite[Theorem 2.3]{ASZ}).
\end{proof}

To generalise the proof of the BF property given for Auslander-Gorenstein rings in Corollary \ref{AG} to the broader setting of algebras having an Auslander dualizing complex we need to be able to precisely control the increase in the value of the grade when passing from a finitely generated module $M$ to $M/\psi(M)$ for certain finitely generated $j-$pure modules $M$ and module monomorphisms $\psi$. Note that this was possible for rather trivial reasons in the proof of Corollary \ref{AG}, for the crucial case there of $R = M$. In the present setting an inequality in one direction follows easily, as we show in Lemma \ref{grog}. First we note in Lemma \ref{drop} that Lemma \ref{grog} applies in the key case where $M = R/P$ for a prime ideal $P$ of $R$.

\begin{lemma}\label{drop}  Let $R$ be a noetherian algebra with an Auslander dualizing complex $\mathcal{R}$ and let $P$ be a prime ideal of $R$. Then $R/P$ is a $j$-pure and $j^{\circ}$-pure $R$-module.
\end{lemma} 
\begin{proof} We prove the result for $j$, the argument for left modules being identical. Suppose the result is false, and let $T$ be a nonzero right ideal of $R/P$ with 
\begin{equation}\label{omega} j (T) >  j (R/P).
\end{equation} 
Replacing $T$ by  a smaller right  ideal if necessary, we may assume that $T$ is uniform. By \cite[Proposition 7.24]{goodearl-warfield04},  there exists an essential right 
ideal $B$ of $A/P$ such that 
$$ B \; \cong \;  \oplus_{i = 1}^{t} T,$$
where $t$ is the uniform dimension of $R/P$. Since $j$ is exact by Theorem \ref{gradedim}, (\ref{omega}) implies that
\begin{equation}\label{first} j(B) \; = \; j(T) \; >  \; j ((R/P) _R).
\end{equation}
But, by Goldie's key lemma, \cite[Proposition 6.13]{goodearl-warfield04}, $B$ contains a regular element $d + P$ of $R/P$. Since $R/P \cong dR + P/P \subseteq B$ we deduce that 
\begin{equation}\label{second}j (B) \; \leq \; j ((R/P) _R).
\end{equation}
(\ref{first}) and (\ref{second}) yield a contradiction, so no such right ideal $T$ of $R/P$ can exist. The argument on the left is the same.
\end{proof}

Recall that a (right, say) module $M$ over a noetherian ring $S$ is called a \defit{torsion} $S$-module if for each $m \in M$ there exists $c \in S^{\bullet}$ such that $mc = 0$. The following result is a minor strengthening of the work in \cite[$\S$2]{YZ}, where the result is noted for the case where $I$ is a prime ideal.

\begin{lemma}\label{grog} Let $R$ be a noetherian $k$-algebra with an Auslander dualizing complex $\mathcal{R}$. Then $j(M) \geq j(R/I) + 1$, for every ideal $I$ such that $R/I$ is $j-$pure and every finitely generated torsion right $R/I$-module $M$; and similarly for $j^{\circ}$ with left modules.    
\end{lemma}

\begin{proof}  By the exactness of $j$ it is enough to prove this when $M = R/cR + I$, where $c \in R$ is such that $c + I \in (R/I)^{\bullet}$. Moreover exactness also implies that 
\begin{equation}\label{fly} j(R/cR + I) \; \geq \; j(R/I), 
\end{equation}
so by Proposition \ref{finite} it remains only to show that equality is impossible in (\ref{fly}). To see this, note that, for all $i \geq 0$,
\begin{equation}\label{down} c^i R + I/c^{i + 1}R + I \; \cong \; R/cR + I.
\end{equation}
However $-j$ is finitely partitive by Theorem \ref{gradedim}, so (\ref{down}) shows that equality in (\ref{fly}) would yield a contradiction.
\end{proof}

In particular, we thus see that $j(R/cR + I) \geq j(R/I) + 1$ when $R, I $ and $c$ are as in Lemma \ref{grog} and its proof. To show that this is in fact an equality seems more tricky, requiring a resort to Gabber's Maximality Principle, which we now recall.

\begin{definition}\label{GabberMax}\cite[page 145]{Bj}
 Let $\delta$ be a dimension function on $\mathrm{Mod}(R)$ and let $N$ be a (not necessarily finitely generated) $\delta-$pure $R$-module with $\delta(N) = n$. Then $\delta$ is said to satisfy \defit{Gabber's Maximality Principle} on  $N$ 
if, for every finitely generated submodule $M$ of $N$,  there exists a submodule $\widetilde {M}$ of $N$  maximal such that $\delta(\widetilde {M} / M ) \leq n - 2$, and $\widetilde {M} $ is finitely generated. 
\end{definition}

The crucial point to observe about the definition is that $N$ is \emph{not} assumed to be finitely generated. Note that when $\delta$ is exact, we can see easily that the module $\widetilde {M}$ in the definition is unique. Here is the result we need about Gabber's Maximality Principle. 

\begin{theorem}\label{compGabber}\cite[Theorem 2.19]{YZ} Let $R$ be a noetherian $k$-algebra with an Auslander dualizing complex $\mathcal{R}$, and let $-j = -j_{\mathcal{R};R}$ be the associated canonical dimension function on $\mathrm{Mod}(R)$. Then $-j$ satisfies Gabber's Maximality Principle on all nonzero $j$-pure $R$-modules. 
\end{theorem}

Now we are in a position to see that every prime factor of a noetherian $k$-algebra with an Auslander dualizing complex is a BF-ring:

\begin{theorem}\label{gradejump} Let $R$ be a noetherian $k$-algebra with an Auslander dualizing complex $\mathcal{R}$, and let $j = j_{\mathcal{R};R}$ be the associated grade function on $R$-modules.
\begin{enumerate}
\item[(1)] If $R/I$ is a $j-$pure or $j^{\circ}-$pure factor ring of $R$ with an artinian classical ring of quotients, then $R/I$ is a BF-ring.
\item[(2)] Let $P$ be a prime ideal of $R$. Then $R/P$ is a BF-ring. 
\end{enumerate}
\end{theorem}

\begin{proof}(1) Let $c \in R$ be such that $c +I$ is a regular nonunit of $R/I$. We claim that  
\begin{equation}\label{key} j(R/cR + I) \; = \; j(R/I) + 1. 
\end{equation}
By Lemma \ref{grog} and Proposition \ref{finite},
\begin{equation}\label{lessthan}  j (R/I) <  j (R/cR + I) <   \omega.
\end{equation}
Suppose for a contradiction that 
\begin{equation}\label{false}   j(R/cR + I) \; \geq  \; j(R/I) +  2. 
\end{equation}
Let $Q(R/I)$ denote the artinian quotient ring of $R/I$. Write $\widehat{c} := c + I \in (R/I)^{\bullet}$, and define 
$$\widehat{M} \; := \;\bigcup_{n \geq 0}\widehat{c}^{-n}R/I \; \subseteq Q(R/I).$$
For each $n \geq 1$, $\widehat{c}^{-n}R/I$ has a finite chain of $R$-submodules $\{M_i := \widehat{c}^{-i}R/I : 0 \leq i \leq n \}$, with $M_0 = R/I$ and 
successive subfactors $M_{i+1}/M_i$ isomorphic to $R/cR + I$ for $0 \leq i < n$. Therefore, by exactness of $j$ and (\ref{false}),
\begin{equation}\label{done} j(\widehat{M}/(R/I)) \; \geq \; j(R/I) + 2. 
\end{equation}
However $R/I$ is $j$-pure by hypothesis. Hence $Q(R/I)$, being a union of copies of $R/I$, is also 
$j$-pure. So Theorem \ref{compGabber} implies that $\widehat{M}/(R/I)$ is a finitely generated $R/I$-submodule
 of $Q(R/I)/(R/I)$. Since $c + I$ is not a unit of $R/I$, this is easily seen to be impossible, so (\ref{false}) is false and 
 \begin{equation} 
    \label{factor} j(R/cR + I) \; = \; j(R/I) + 1. 
    \end{equation}
The result now follows at once from Lemma \ref{easyandimportant}. 

(2) This is a special case of (1), since $R/P$ is $j-$pure by Lemma \ref{drop} and $R/P$ has a simple artinian quotient ring by Goldie's theorem.
\end{proof}  

\medskip

\begin{remarks}\label{dualexs}(1) The question of which affine noetherian $k$-algebras have (Auslander) dualizing complexes is extensively discussed in \cite{Van}, \cite{YZ}. For example, \cite[Corollary 6.8]{YZ} states (roughly) that an $\mathbb{N}-$filtered $k$-algebra, with $A_0 = k$ and $\mathrm{dim}_k(A_n) < \infty$ for all $n$, whose associated graded algebra is noetherian and has a (graded) Auslander dualizing complex will itself have an Auslander dualizing complex.
\end{remarks} 

(2) It is a consequence of \cite[Corollary 2.18]{YZ} that a noetherian $k$-algebra which has an Auslander dualizing complex must have finite (Gabriel-Rentschler) Krull dimension. This suggests a possible direction in which to look for a noetherian algebra \emph{not} satisfying property BF.

(3) If one is faced in the setting of Theorem \ref{gradejump}(1) with an ideal $I$ of the algebra $R$ which is \emph{not} semiprime, it may in practice be difficult to determine whether the hypotheses of $j-$purity of $R/I$, and the existence of an artinian quotient ring of $R/I$, are satisfied. In fact, it may be that $j-$purity \emph{implies} the existence of an artinian quotient ring. 

Here is an important case where this is true, and moreover where $j-$purity may become much easier to determine. Let $R$, $\mathcal{R}$ and $j= j_{\mathcal{R};R}$ be as in Theorem \ref{gradejump}, and suppose that $R$ has finite Gelfand-Kirillov dimension, $\mathrm{GKdim}(R) = n < \infty$. (See \cite{GW} for the properties of Gelfand-Kirillov dimension.) Following \cite[Definition 5.8]{Lev} and the generalisation in \cite[Definition 2.24]{YZ} we say that $R$ is \defit{GK-Cohen Macaulay with respect to} $j_{\mathcal{R};R}$ if, for every nonzero finitely generated $R$-module $M$,
\begin{equation}\label{CM} \mathrm{GKdim}(M) + j_{\mathcal{R};R}(M) \; = \; n. 
\end{equation}

In the setting of Theorem \ref{gradejump} we know by Theorem \ref{gradedim} that $j$ is exact, so when $R$ is also GK-Cohen Macaulay the Gelfand-Kirillov dimension is also exact, by (\ref{CM}). Moreover (\ref{CM}) also shows that the $j-$purity of $R/I$ is equivalent to GK-purity of $R/I$; usually, we say then that $R/I$ is \defit{GK-homogeneous}. Finally, when GK-dimension is exact, a GK-homogeneous noetherian algebra has an artinian quotient ring, by \cite[Theorem 5.4]{Lev}. We have thus obtained from Theorem \ref{gradejump}(1):

\begin{corollary}\label{GKBF} Let $R$ be a noetherian $k$-algebra of finite GK-dimension which has an Auslander dualizing complex $\mathcal{R}$. Suppose that $R$ is GK-Cohen Macaulay with respect to $j_{\mathcal{R};R}$, and let $I$ be an ideal of $R$ with $R/I$ GK-homogeneous. Then $R/I$ is a BF-ring.
\end{corollary}

Corollary \ref{GKBF} prompts the following 

\begin{questions} (1) Does every affine noetherian $k$-algebra of finite Gelfand-Kirillov dimension have an Auslander dualizing complex?

\noindent (2) Is every affine noetherian $k$-algebra of finite Gelfand-Kirillov dimension a BF-ring?  
\end{questions}
This would provide a major generalization of the results from \cref{sec:quadratic}.

\medskip

(4) Regarding Remark \ref{dualexs}(3), we might hope that, given a noetherian $k$-algebra $R$ with finite and exact GK-dimension, property BF 
for $R$ could be approached by applying Lemma \ref{easyandimportant} directly, using $\mathrm{GKdim}$ as the function $j$ of that lemma. However it is not clear that the Gelfand-Kirillov dimension will 
be finitely partitive on $R$-modules, nor that $\mathrm{GKdim}(R/\alpha R) = \mathrm{GKdim}(R) - 1$ for $\alpha \in R^{\bullet} \setminus R^{\times}$. Recall that, when $\mathrm{GKdim}$ is replaced by 
the Krull dimension $\mathrm{Kdim}$, the latter property fails drastically for the Weyl algebras: in a famous paper \cite{St}, Stafford showed that, for every $n \geq 2$, the Weyl algebra $A_n 
(\mathbb{C})$ contains an element $c$ and $d$ such that $A_n (\mathbb{C})/c A_n(\mathbb{C})$ is a simple module and so has Krull dimension $0$, 
whereas $\mathrm{Kdim}(A_n (\mathbb{C})/dA_n(\mathbb{C})) = n - 1$. Since the Weyl algebras are Auslander regular and GK-Cohen Macaulay, we know from (the proof of) Corollary \ref{AG} that
$$\mathrm{GKdim}(A_n (\mathbb{C})/c A_n(\mathbb{C})) = \mathrm{GKdim}(A_n (\mathbb{C})/d A_n(\mathbb{C})) = 2n -1,$$
but this can be be proved directly, much more easily - see \cite[Corollary 8.6]{GW}.

\bigskip

\subsection{Commutative noetherian rings revisited}\label{commagain} Recall that the \defit{support} of a module $M$ over a commutative ring $R$ is the set $\mathrm{supp}(M)$ of prime ideals of $R$ such that $M_P \neq 0$. It is easy to see that if $M$ is finitely generated, $\mathrm{supp}(M) = \{P : \mathrm{Ann}(M) \subseteq P \}$, where $\mathrm{Ann}(M)$ denotes the annihilator of $M$. Let $\mathrm{ht}(P)$ denote the height of a prime ideal $P$ of $R$. Consider the following definition:

\begin{definition}\label{suppheight}Let $R$ be a commutative noetherian ring. Define a function $j$ on $\mathrm{Mod}_f (R)$ by
$$ j(M) \; := \; \mathrm{inf}\{ \mathrm{ht}(P) : P \in \mathrm{supp}(M) \}, $$
with $j(0) := \infty$. 
\end{definition}

It is easy to check that the above map $j$ is exact on $\mathrm{Mod}_f (R)$ - this is an immediate consequence of the exactness of localisation. Thus $j$ can be extended to a map from $\mathrm{Mod}(R)$ to $\mathbb{Z}_{\geq 0} \cup \{\infty\}$ as discussed after Notation \ref{notation}. We can give yet another proof of 

\begin{proposition}\label{commmore} If $R$ is a commutative noetherian ring then $R$ is a BF-ring.
\end{proposition}

\begin{proof} This is a consequence of Lemma \ref{easyandimportant} provided it can be shown that the above map $j$ is (i) finitely partitive on $\mathrm{Mod}_f (R)$, and (ii) $j(R/xR) = 1$ for all $x \in R^{\bullet} \setminus R^{\times}$.

To prove (i) let $M$ be a nonzero finitely generated $R$-module with $j(M) = n$, and let $\{P_1, \ldots , P_t \}$ be the subset of primes containing $\mathrm{Ann}(M)$ of height $n$. Note that this set is finite, since it consists of primes minimal over $\mathrm{Ann}(M)$. For each $i = 1, \ldots , t$, $M_{P_{i}}$ is a nonzero $R_{P_{i}}$-module of finite composition length, say $\ell_i$. Now it is easy to see that any chain of $R$-submodules of $M$ with all subfactors $X$ having $j(X) = n$ must have length at most $\sum_i \ell_i$, as required. 

Since (ii) is an immediate consequence of the Principal Ideal Theorem and the fact that minimal primes of $R$ consist of zero divisors, \cite[Theorem 10.1 and Corollary 2.18]{Eis}, this completes the proof.  
\end{proof}

Note however that this is essentially a rephrasing of the classical proof \cite[Proposition 2.2]{anderson-anderson-zafrullah90} into the present setting.

\bigskip

The following example shows that, while the function $j$ of Definition \ref{suppheight} satisfies the hypotheses of Lemma \ref{easyandimportant}, $-j$ does \emph{not} satisfy Gabber's Maximality Principle.

\begin{example}\label{melexample}(M. Hochster, R. Heitmann) The ring  $S = K\llbracket x,y,z\rrbracket /(x^2, xy)$, where $K$ is a field  of characteristic zero,  is the completion of a noetherian domain $R$. The function $-j$ described after the Example \ref{8.6}
 does not satisfy 
Gabber's Maximality Principle on $Q_R$, with $Q$ the denoting the quotient field of $R$. 

\medskip
The proof of this claim is organized in four steps.

\smallskip
\noindent
\textbf{Step (1)}: \emph{$S$ is a local noetherian complete ring of dimension $2$.}

The formal power series ring  $k\llbracket x, y, z\rrbracket$ is a local noetherian complete ring of dimension $3$ (see \cite [Exercise 15.30]{Sh00}). 
Since a quotient of a complete ring is again complete, the ring $S$ is complete, and due to the chain
\[
  (x)/ (x^2, xy)  \subsetneq  (x, y) / (x^2, xy) \subsetneq (x, y, z) / (x^2, xy) \subsetneq S
\]
of prime ideals, 
 $S$ is of dimension $2$.  

 \smallskip
\noindent
 \textbf{Step (2)}: \emph{$S$ is the completion of a noetherian local domain, say $R$.}
 
We must check the two conditions in \cite [Theorem 1] {Lech}. Thus we have to show:
(i) the prime ring of $S$, that is the set $\{\, k1_S : k \in \Bbb{Z} \,\}$, is a domain that acts on $S$ without torsion; 
(ii) the maximal ideal of $S$ does not belong to the set of associated prime ideals of $0$.

Since $S$ is an $K$-algebra and $K$ is of characteristic zero, condition (i) is trivially satisfied.
To verify (ii), let $0 = Q_1 \cap \cdots \cap Q_n$ be a primary decomposition of $0$. The associated prime ideals of $0$ are the radicals $P_i$ of $Q_i$. The union $P_1 \cup \cdots \cup P_n$ is the set of zero-divisors of $S$.
Since $\bar {z} \in (x, y, z)/(x^2, xy)$ is not a zero-divisor, the maximal ideal $(x, y, z)/ (x^2, xy)$ 
is not an associated prime ideal of $0$. Therefore (ii) holds too.

\smallskip
\noindent
\textbf{Step (3)}:  \emph{$R$ is of dimension $2$ and $H ^1_{\fm_R} (R)$,  that is, the direct limit of the direct system $\{\, \Ext^1( R/ \fm_R^k , R) : k \geq 1 \,\}$,  is not a finitely generated $R$-module. \textup{(}Here $\fm_R$ denotes the maximal ideal of $R$.\textup{)}}

Krull dimension is preserved by passing to the completion. Therefore $R$ has Krull dimension $2$.
There is a natural embedding from $R$ into its completion $S$ and the maximal ideal of $S$, that is $\fm_S$, is equal to $\fm_R S$. 
We have an $S$-isomorphism $H^1 _{\fm_R} (R) \cong  H^1_{\fm_S} (S) $, see \cite [Lemma 3.5.4(d)]{BRHE98}.
To show that $H^1_{\fm_R} (R)$ is not finitely generated as $R$-module, it is therefore enough to show that $H^1 _{\fm_S} (S) $ is not finitely generated as an $S$-module.

Suppose that $T = K\llbracket x,y,z\rrbracket$.
Then $T$ is a regular local ring  (and so Gorenstein and Cohen-Macaulay) of dimension $3$.
We can apply 
\cite [Theorem 3.3.7]{BRHE98} and Grothendieck's duality (local duality theorem) \cite [Corollary 3.5.9] {BRHE98} to  
$M = S$ to see that $H^1_{\fm_T}(S)$ is the Matlis Dual of $\Ext^2_T (S,T)$.
Now by \cite [Theorem 3.5.4(a)] {BRHE98}, $H^1 _{\fm_S}(S)$ is artinian and also equal to 
$H^1_{\fm_T} (S)$, see \cite [Corollary A1.8] {Apendix}.
So if $H^1 _{\fm_S} (S) $ is a finitely generated $S$-module, so is  $H^1 _{\fm_T} (S) $  and so they are of finite length.
This shows that its dual, that is $\Ext^2 _T (S, T)$, is of finite length, see \cite [Theorem 3.2.13] {BRHE98}. 
Thus $\Ext^2_T(S,T)$ vanishes when one localizes at $P = (x,y)T$.  
Localization commutes with Ext for finitely generated
modules over commutative noetherian rings (\cite [Prop 3.3.10]{weibel}), and this would imply $\Ext^2_{T_P}(S_P, T_P) = 0$.
Again using the duality theorem, since $T_P$ is a regular ring of Krull dimension $2$,  the Matlis dual of $\Ext^2_{T_P}(S_P, T_P)$ over $T_P$ is $H^0_{PT_P}(S_P)$ (this is the direct limit of direct system  $\{ \Hom( T_P / (PT_P) ^k , S_P ) : k\geq 0 \}$).
But the latter can not be zero because $H^0_{PT_P}(S_P)$ is the set of all elements of $S_P$ that are annihilated by a power of $PT_P$, and the image of $x$ in $S_P$ is a nonzero element of $H^0_{PT_P}(S_P)$  (see for example \cite [Page 126, 127] {BRHE98} or  \cite [Page 187]{Apendix}). 

\smallskip
\noindent
\textbf{Step (4)}:  \emph{$H ^1 _{\fm_R} (R) $ is isomorphic to a submodule of $Q/R$ and the minimal prime ideal containing the annihilator of $H ^1 _{\fm_R} (R)$ is $\fm_R$.}

Every noetherian local ring of dimension $d$ has a system of parameters of $d$ elements.
Since $R$ is of dimension $2$, this means that there exist $u$,~$v \in R$ such that $\fm_R$ is the minimal prime ideal over $(u, v)$. 
Then it is known that $H^1 _{\fm_R} (R) = H^1_{(u, v)} (R) $, as the radical of $(u, v)$ is $\fm_R$.
Consider the complex
\[
  0 \to R \overset{f} \longrightarrow  R[u^{-1}] \oplus R[v^{-1}]  \overset{g} \longrightarrow R[(uv)^{-1}] \to 0,
\]
where $f$  maps  $r$  to   $(r,r)$  and $g$ takes  $(r/u^s, r'/v^t)$ to $r/u^s  - r'/v^t$. 
By \cite [Theorem A1.3]{Apendix}, the cohomology $H^1_{(u,v)}(R)$ is isomorphic to $\ker(g)/\im(f)$.
Noting that $\ker(g)$ can be identified with a submodule of $Q$, this  means that $H^1_{\fm_R}(R)$ is isomorphic to a submodule of $Q/R$.

For the other part of Step (4), note that if $k + \im(f)$ is an element of $\ker(g) / \im(f)$, there exist $n$, $m \geq 1$ such that $u^n$, $v^m$ are in the annihilator of $k + \im(f)$.
Therefore a minimal prime ideal containing the annihilator of $k + \im(f)$ must contain $u$ and $v$.
Since $\fm_R$ is the minimal prime containing $(u, v)$, this implies that the minimal prime ideal containing $\ann(k+ \im(f))$ is $\fm_R$.  
Similarly the minimal prime ideal of the annihilator of  every finitely generated submodule of $\ker(g) / \im(f)$ is 
$\fm_R$, which is of height two. 
\end{example}

\section{An example} \label{s:exm-non-bf}

In this section we construct a finitely presented atomic domain that does not have BF.
In fact, it will be an atomic domain that does not have ACCP.
Our example is a semigroup algebra (over a field), and the main difficulty lies in establishing its atomicity.

Let $K$ be a field and $S$ a monoid.
We write $K[S]$ for the semigroup algebra of $S$ over $K$.
Every $f \in K[S]$ has a unique representation of the form
\[
  f = \sum_{s \in S} a_s s \qquad\text{with $a_s \in K$, almost all zero}. 
\]
We write $\supp(f) \coloneqq \{\, s \in S : a_s \ne 0 \,\}$ for the support of $f$.

In the noncommutative setting, in general, the problem of characterising (semi)group algebras that are domains is a very hard one, known as Kaplansky's zero-divisor conjecture \cite[Chapter 13]{passman77}, as is the related characterisation of units (with a recent counterexample to the unit conjecture given by Gardam \cite{gardam21}).
To avoid these issues, we work with right orderable monoids.

A monoid $S$ is \defit{right orderable} if there exists a total ordering on $S$, such that $a < b$ implies $ac < bc$ for all $a$,~$b$,~$c \in S$.
Recall that a submonoid $S \subseteq T$ is \defit{divisor-closed} if, whenever $a$,~$b \in T$  are such that $ab \in S$, then already $a \in S$ and $b \in S$.

\begin{lemma} \label{l:ordered}
  Let $K$ be a field and let $S$ be a right orderable cancellative monoid.
  \begin{enumerate}
  \item\label{ordered:divclosed} $K^{\times} S$ is a divisor-closed submonoid of $K[S]$.
  \item\label{ordered:domain} $K[S]$ is a domain.
  \item\label{ordered:units}  $K[S]^\times = K^\times S^\times$.
  \end{enumerate}
\end{lemma}

\begin{proof}
  As in the case where $S$ is a group \cite[Lemma 13.1.7 and 13.1.9]{passman77}.
\end{proof}

We now construct the example.
For the rest of this section, let $F$ be the free group on generators $b$,~$c$.
Let $\alpha$ be the group automorphism of $F$ defined by $\alpha(b) = c$ and $\alpha(c)=b^{-1}$.
Consider the semidirect product $G \coloneqq F \rtimes_{\alpha} \operatorname{gr}(a)$, where $(\operatorname{gr}(a), \cdot) \cong (\bZ,+)$ is infinite cyclic.
We identify $a=(1,a)$, $b=(b,1)$, and $c=(c,1)$.
Then $G$ is the group generated by $a$, $b$, $c$, with relations generated by $aba^{-1}=c$ and $aca^{-1}=b^{-1}$.
Let $S \subseteq G$ be the submonoid generated by $a$ and $b$.

\begin{lemma} \label{l:exm-g}
  \begin{enumerate}
  \item \label{exm-g:order} $G$ is right-orderable.
  \item \label{exm-g:bf} $K[G]$ is a BF-domain.
  \end{enumerate}
\end{lemma}

\begin{proof}
  \ref{exm-g:order}
  Since both $F$ and $\bZ$ are right orderable (indeed, both of them are even orderable), the semidirect product $G$ is right orderable as well \cite[Lemma 13.1.5]{passman77}.

  \ref{exm-g:bf}
  $K[F]$ is a free ideal ring by a result of P.~M.~Cohn \cite[Corollary 7.11.8]{cohn06}.
  By \cite[Proposition 3.2.9]{cohn06}, this implies that $K[F]$ is atomic and that a strong uniqueness property holds for the factorizations of an element $x \in K[F]$: given any two factorizations of $x$, it is possible to pass from one to the other using a series of comaximal transpositions (essentially an application of Jordan-Hölder to $[aR,R]$, which is a modular lattice of finite length thanks to $R$ being a free ideal ring; see the discussion preceding \cite[Proposition 3.2.9]{cohn06}).
  In particular, the length of two factorizations of $a$ is therefore the same.
  Thus, $K[F]$ is half-factorial and, in particular, has BF.
  Since $K[G]$ is a skew Laurent polynomial ring over $K[F]$, also the ring $K[G]$ has BF by \cref{p:skew}.
\end{proof}

Our next step is to exhibit an explicit presentation of the submonoid $S$ of $G$ in terms of generators and relations.

\begin{lemma} \label{l:ex-grp}
  $S$ is isomorphic to $\langle a,b \mid ba^2b = a^2,\ a^4b=ba^4 \rangle$.
\end{lemma}

\begin{proof}
  One checks immediately that $ba^2b=a^2$ and $a^4b=ba^4$ hold in $G$.
  We claim that, using these two relations, any word $x$ in $a$,~$b$ can be reduced to the form
  \begin{equation} \label{eq:nf}
    x = b^{m_0}a^{n_1}b^{m_1}a^{n_2}b^{m_2} \cdots a^{n_k}b^{m_k} a^n,
  \end{equation}
  with $k \ge 0$, $m_0 \ge 0$, $m_1$,~$\ldots\,$,~$m_k > 0$, $n \ge 0$, $n_2$, $\ldots\,$,~$n_k \in \{1,3\}$ and $n_1 \in \{1,2,3\}$.
  Further, if $n_1=2$ we can assume $m_0=0$.
  
  Indeed, a priori $x$ has such a form with $m_0$,~$n \ge 0$, $k \ge 0$, $m_1$,~$\ldots\,$,~$m_k >0$ and $n_1$, $\ldots\,$,~$n_k > 0$.
  Moving fourth powers of $a$ to the right using $a^4b=ba^4$, we may assume $n_i \in \{1,2,3\}$ for all $i \in \{1, \ldots,k\}$.
  If $i \ge 2$ and $n_i=2$ we may successively use $ba^2b=a^2$ to reduce $m_{i-1}$ and $m_{i}$ until one of them becomes zero, at which point we obtain a representation with smaller $k$ and continue inductively.
  This leaves only $n_1$ to deal with.
  If $m_0 \ge m_1$, we can again merge $a^{n_1}$ with $a^{n_2}$.
  So if $n_1=2$ we can assume $m_0<m_1$ and then reduce to $m_0=0$.

  We now claim that any two words of the form above with distinct parameters yield distinct elements of $G$.
  To deduce this, we reduce a word in the form \eqref{eq:nf} to the normal form in $G$ in which all $a$'s are on the right.
  Note that $ab=ca$, $ac=b^{-1}a$, $a^3b = c^{-1}a^3$, and $a^3c = b a^3$.

  Suppose first $n_1 \ne 2$.
  Let $\varepsilon_i = 1$ if $n_i=1$ and $\varepsilon_i = -1$ if $n_i = 3$ for $1 \le i \le k$.
  Then $a^{n_i}b^m = c^{\varepsilon_i m} a^{n_i}$ and $a^{n_i}c^m = b^{-\varepsilon_i m} a^{n_i}$ for all $1 \le i \le k$ and $m \ge 0$.
  Inductively, one obtains
  \[
    a^{n_1+\cdots + n_i} b^m =
    \begin{cases}
      b^{\varepsilon_1\cdots \varepsilon_i m} a^{n_1+\cdots+n_i}  & \text{if $i \equiv 0 \mod 4$,} \\
      c^{\varepsilon_1\cdots \varepsilon_i m} a^{n_1+\cdots+n_i}  & \text{if $i \equiv 1 \mod 4$,} \\
      b^{-\varepsilon_1\cdots \varepsilon_i m} a^{n_1+\cdots+n_i}  & \text{if $i \equiv 2 \mod 4$,} \\
      c^{-\varepsilon_1\cdots \varepsilon_i m} a^{n_1+\cdots+n_i}  & \text{if $i \equiv 3 \mod 4$.} \\
    \end{cases}
  \]
  Computing in $G$, we thus find
  \[
    \overline{x} = b^{m_0} c^{\varepsilon_1 m_1} (b^{-1})^{\varepsilon_1 \varepsilon_2 m_2}(c^{-1})^{\varepsilon_1 \varepsilon_2 \varepsilon_3 m_3}b^{\varepsilon_1\varepsilon_2\varepsilon_3\varepsilon_4 m_4} \cdots e^{\varepsilon_1 \cdots \varepsilon_{k-1} m_{k-1}} d^{\varepsilon_1 \cdots \varepsilon_k m_k} a^{n_1 + \cdots + n_k+n},
  \]
  where $(e,d) \in \{ (c^{-1},b), (b,c), (c,b^{-1}), (b^{-1},c^{-1}) \}$ according to the congruence class of $k$ modulo $4$.
  Note that from the word in $b$,~$c$, working our way from left to right, we can read off all $m_i$'s and $\varepsilon_i$'s, and therefore also the $n_i$'s.
  From the exponent of $a$ we can then compute $n$.

  Suppose now $n_1 = 2$. Then $m_0=0$.
  Let $\varepsilon_i = 1$ if $n_i=1$ and $\varepsilon_i = -1$ if $n_i = 3$ for $2 \le i \le k$.
  We use $a^2b = b^{-1}a^2$ and $a^2 c = c^{-1} a^2$ together with the computation from the first case to find
  \[
    \begin{split}
    \overline{x} &= a^2 b^{m_1}c^{\varepsilon_2 m_2}(b^{-1})^{\varepsilon_2\varepsilon_3m_3} \cdots e^{\varepsilon_2 \cdots \varepsilon_{k-1} m_{k-1}} d^{\varepsilon_2 \cdots \varepsilon_k m_k} a^{n_2 + \cdots + n_k+n}\\
    &= b^{-m_1}c^{-\varepsilon_2 m_2}(b^{-1})^{-\varepsilon_2\varepsilon_3m_3} \cdots e^{-\varepsilon_2 \cdots \varepsilon_{k-1} m_{k-1}} d^{-\varepsilon_2 \cdots \varepsilon_k m_k} a^{2 + n_2 + \cdots + n_k+n} \in G,
    \end{split}
  \]
  where again the values of $d$ and $e$ depend on $k-1$ modulo $4$.
  Note that this case is distinguishable from the previous one because the exponent of the left-most $b$ is negative.
  As before, the $m_i$'s and $n_i$'s can be recovered unambiguously from the representation.
\end{proof}

\begin{lemma} \label{l:infinite-b}
  \begin{enumerate}
  \item \label{infb:monomial}
    If $x \in \bigcap_{n \ge 0} Sb^n$, then there exist $i \ge 0$ and $x_i \in S$ such that $x=x_ia^2b^i$.
    If $x \in \bigcap_{n\ge 0} b^nS$, then there exist $i \ge 0$ and $x_i \in S$ such that $x=b^ia^2x_i$.
  \item \label{infb:polynomial}
  
  Let $f \in K[S]$.
  If $f \in \bigcap_{n \ge 1} K[S]b^n$, then there exist $m \ge 0$ and $g_m \in K[S]$ such that $f=g_ma^2b^m$.
  If $f \in \bigcap_{n \ge 1} b^n K[S]$, then there exist $m \ge 0$ and $g_m \in K[S]$ such that $f=b^ma^2g_m$.
  \end{enumerate}
\end{lemma}

\begin{proof}
  \ref{infb:monomial} By symmetry it suffices to show one of the claims.
  We shall use the normal form established in the proof of \cref{l:ex-grp}.
  Due to the choices made there, it is easier to prove the claim for right ideals.
  So let $x \in \bigcap_{n \ge 0} b^n S$, with normal form and notation as in \eqref{eq:nf}.
  For sake of contradiction, assume $x$ is not of the form $b^ia^2 x_i$ with $i \ge 0$ and $x_i \in S$.
  In the normal form for $x$, we must then have $n_1 \in \{1,3\}$.
  Suppose $x = b^ny$ for some $n > m_0$ and $y$ in $S$.
  Writing $y$ in normal form, we have
  \[
    y = b^{\mu_0}a^{\nu_1}b^{\mu_1}a^{\nu_2} b^{\mu_2}\cdots a^{\nu_l}b^{\mu_l}a^{\nu},
  \]
  with $l \ge 0$, $\mu_0 \ge 0$, $\mu_1$, $\ldots\,$,~$\mu_l > 0$, $\nu \ge 0$, $\nu_2$, $\ldots\,$,~$\nu_l \in \{1,3\}$ and $\nu_1 \in \{1,2,3\}$.
  First of all, note that $\nu_1=2$ leads to a contradiction to our choice of $x$, so that $\nu_1 \in \{1,3\}$.
  However, now $b^ny$, with $y$ expressed in its normal form, is already in normal form.
  But $x=b^ny$ and, comparing the normal forms, $m_0=n+\mu_0$ yields a contradiction to $n > m_0$.

  \ref{infb:polynomial}
  By symmetry, it suffices to prove the first claim.
  If $f \in K[S] b^n$, then each monomial in $\supp(f)$ is in $S b^n$.
  Let $f \in \bigcap_{n \ge 0} K[S] b^n$.
  By \ref{infb:monomial}, we may write
  \[
    f = \sum_{j=0}^k \lambda_{i_j} x_{i_j} a^2 b^{i_j},
  \]
  with suitable $k \ge 0$, $i_j \ge 0$, $x_{i_j} \in S$, and $\lambda_{i_j} \in K$.
  Choose $m \ge \max\{i_1,\ldots,i_k\}$.
  Then, by using $a^2=ba^2b$,
  \[ 
  f = \sum_{j=0}^k \lambda_{i_j} x_{i_j} b^{m-i_j} a^2 b^m = \Big( \sum_{j=0}^k \lambda_{i_j} x_{i_j} b^{m-{i_j}}\Big) a^2 b^m. \qedhere
  \]
\end{proof}

For $0 \ne f \in K[G]$ we let $\deg_a(f) \in \bZ$ be the maximal $a$-degree of any monomial in the support of $f$, and we set $\deg_a(f)=-\infty$ if $f=0$.
This is a well-defined degree function, since $K[G]$ may be viewed as a skew Laurent polynomial ring in the indeterminate $a$ over $K[F]$.
Observe $\deg_a(f) \in \bZ_{\ge 0} \cup \{-\infty\}$ for all $f \in K[S]$.

\begin{lemma}
  $K[S]$ is atomic.
\end{lemma}

\begin{proof}
  Let us say $f \in K[S]^\bullet$ is \defit{atomic} if $f$ can be represented as a product of atoms or if $f$ is a unit.
  Since $K[b] \subseteq K[S]$ is a divisor-closed submonoid, and $K[b]$ is a polynomial ring, every $f \in K[b]^\bullet$ is atomic in $K[S]$.
  We first show:
  \begin{enumerate}[leftmargin=2cm]
  \item[\textbf{Claim A.}] for every nonunit $f \in K[S]^\bullet$, there exists an atomic nonunit $g \in K[S]^\bullet$ such that $f \in gK[S]$.
  \end{enumerate}
  Suppose that this is not the case, and let $m \in \bZ_{> 0}$ be the minimal $a$-degree among all counterexamples.
  Let $\Omega \subseteq K[S]$ be the set of all $f \in K[S]^\bullet \setminus K[S]^\times$ with $\deg_a(f)=m$ and such that $f$ does not have a nonunit atomic left factor.
  Since $K[G]$ has BF (by \cref{l:exm-g}\ref{exm-g:bf}), it satisfies the ascending chain condition on principal right ideals.
  Hence $\{\, fK[G] : f \in \Omega \,\}$  has a maximal element $fK[G]$ with $f \in \Omega$.

  If $f \in \bigcap_{n \ge 1} K[S] b^n$, then $f = ga^2b^l$ with $g \in K[S]$ and $l \ge 0$ by \cref{l:infinite-b}.
  Then $\deg_a(g) < m$. If $g$ is a nonunit of $K[S]$, then $g$ has a nonunit atomic left factor by choice of $m$.
  If $g$ is a unit, then $a$ is a nonunit atomic left factor of $f$ (note that $a \in K[S]$ is an atom). In either case we arrive at a contradiction to our assumption on $f$.
  Thus there exists a maximal $n \ge 0$ such that $f=f'b^n$ with $f' \in K[S]$.
  Then $fK[G]=f'K[G]$.
  Replacing $f$ by $f'$ we may assume $f \not \in K[S]b$.
 
  By construction, the element $f$ cannot be an atom, and thus $f=gh$ with $g$,~$h \in K[S]^\bullet \setminus K[S]^\times$.
  If $\deg_a(g)$,~$\deg_a(h) < m$, then we obtain a nonunit atomic left factor of $f$ from $g$ (if $g \not \in K[S]^\times$) or $h$ (if $g \in K[S]^\times$).
  Thus either $\deg_a(g)=m$ or $\deg_a(h)=m$.

  Suppose first $\deg_a(h)=m$.
  Then $\deg_a(g)=0$ implies $g \in K[b]$.
  Therefore $g$ is atomic in $K[S]$, a contradiction.

  Let now $\deg_a(g)=m$.
  Since $g$ is not atomic, the maximality of $fK[G]$ implies $fK[G]=gK[G]$.
  Thus $h \in K[G]^\times \cap K[b]$, and therefore $h = \lambda b^n$ with $\lambda \in K^\times$ and $n \ge 1$.
  But this contradicts $f \not \in K[S]b$.

  \smallskip
  Having shown \textbf{Claim A}, we can now show that $K[S]$ is atomic.
  Assume that this is not the case, and let $m \ge 0$ be the smallest $a$-degree among non-atomic elements of $K[S]^\bullet$.
  Then
  \[
    \Omega = \{\, K[G] f : f \in K[S] \text{ is not atomic and $\deg_a(f)=m$} \,\}
  \]
  has a maximal element $K[G]f$.
  If $f \in \bigcap_{n \ge 1} b^n K[S]$, then $f=b^ma^2g$ for some $m \ge 0$ and $g \in K[S]$.
  Since $\deg_a(g) < \deg_a(f)$, then $g$ is atomic, and so is $f$.
  Thus there exists again a maximal $n \ge 0$ such that $f = b^n f'$ with $f' \in K[S]$, and we replace $f$ by $f'$ to obtain $f \not \in b K[S]$.
  By the claim, there exists an atomic nonunit $g \in K[S]^\bullet$ such that $f=gh$ with $h \in K[S]$.
  Then $h$ cannot be atomic and so $\deg_a(h)=m$.
  Maximality of $K[G]f$ gives $g \in K[G]^\times \cap K[b]$, and thus $g=\lambda b^n$ with $n\ge 0$.
  But then $n=0$ contradicts $g$ being a nonunit.
\end{proof}

\begin{proposition}
  $K[S]=K\langle a,b \mid ba^2b=a^2,\ a^4b=ba^4 \rangle$ is a finitely presented atomic domain that does not satisfy the ACC on principal right \textup{[}left\textup{]} ideals.
  In particular, the domain $K[S]$ does not have BF.
\end{proposition}

\begin{proof}
  We have already established that $K[S]$ is a finitely presented atomic domain.
  
  Consider the chain of principal right ideals
  \begin{equation} \label{eq:asc}
    a^2K[S] \subseteq ba^2K[S] \subseteq b^2 a^2 K[S] \subseteq \cdots \subseteq b^k a^2 K[S] \subseteq \cdots.
  \end{equation}
  The stated inclusions hold because $b^la^2=b^ka^2b^{k-l}$ for all $k \ge l$.
  Suppose $l > k$ and $b^la^2 = b^ka^2 f$ with $f \in K[S]$.
  Without restriction $k=0$.
  Moreover $f$ must be a monomial with $\deg_a(f)=0$, so that in fact $f = b^m$ for some $m \ge 0$.
  Thus $b^l a^2 = a^2 b^m$ with $l \ge 1$ and $m \ge 0$.
  But this is impossible because the left and the right side are both in the normal form as in \eqref{eq:nf}.
  Thus the chain in \eqref{eq:asc} is an infinite proper ascending chain of principal right ideals.
  By symmetry, $K[S]$ also does not satisfy the ACC on principal left ideals.

  Since $K[S]$ does not satisfy the ACC on principal right [left] ideals, it is in particular not a BF-domain.
  Alternatively, this can be seen directly as follows.
  The monoid $S$ does not have BF as $a$ and $b$ are atoms and $a^2=b^na^2b^n$ implies $\sL(a^2) \supseteq \{\, 2 + 2n : n \ge 0 \,\}$.
  \footnote{In fact, the only factorizations of $a^2$ are those of the form $a^2=b^na^2b^n$, so that even $\sL(a^2) = \{\, 2 + 2n : n \ge 0 \,\}$ holds.
  To verify this, consider the normal forms of monomials containing the element $a$ exactly twice.}
  Since $S$ is divisor-closed in $K[S]$, also $K[S]$ does not have BF.
\end{proof}

Since there exist (commutative) domains that satisfy the ACCP but do not have BF, it would be interesting to know if the previous example can be refined in this direction. We therefore pose the following question.

\begin{question}
  Does there exist a domain $R$ that is a finitely presented algebra over a field, such that $R$ satisfies the ACCP but is not a BF-domain?
\end{question}

\section{Length-preserving homomorphisms to Krull monoids}

Let $H$ and $D$ be monoids.
We shall call a monoid homomorphism $\varphi\colon H \to D$
\begin{itemize}
\item \defit{length-preserving} if $\sL(a) \subseteq \sL(\varphi(a))$ for all $a \in H$;
\item \defit{fully length-preserving} if $\sL(a)=\sL(\varphi(a))$ for all $a \in H$.
\end{itemize}
Note that $a \in H$ is an atom if and only if $1 \in \sL(a)$.
Thus, if $\varphi\colon H \to D$ is length-preserving and $u \in H$ is an atom, then $\varphi(u)$ is an atom of $D$.
Conversely, suppose $\varphi$ maps atoms of $H$ to atoms of $D$.
If $a=u_1\cdots u_k$ with atoms in $H$, then $\varphi(a) = \varphi(u_1)\cdots \varphi(u_k)$ with $\varphi(u_1)$, $\ldots\,$, $\varphi(u_k)$ atoms of $D$.
So, a monoid homomorphism $\varphi\colon H \to D$ is length-preserving if and only if it maps atoms of $H$ to atoms of $D$.

Suppose that $H$ is atomic.
If a length-preserving homomorphism $\varphi\colon H \to D$ to a monoid $D$ with bounded factorizations exists, then $H$ has bounded factorizations as well.
So a useful way to establish that a noetherian prime ring $R$ has bounded factorizations is to find a length-preserving monoid homomorphism $R^\bullet \to D$ to some monoid with bounded factorizations.

In the study of non-unique factorizations transfer homomorphisms to (commutative) Krull monoids play a major role; a monoid possessing such a homomorphism is called a transfer Krull monoid (see \cite[Section 2.4]{geroldinger-zhong19} and \cite[Section 5]{geroldinger-zhong20}).
Large classes of rings, including noncommutative rings, whose underlying multiplicative monoid of cancellative elements are transfer Krull monoids are known \cite[Example 5.4]{geroldinger-zhong20}.
Every transfer homomorphism is fully length-preserving, so for transfer Krull monoids the study of sets of lengths reduces to that of an associated Krull monoid, where a well-understood machinery is available.
In particular, transfer Krull monoids have bounded factorizations.
So it may be interesting to ask whether any of the classes of rings studied in this paper are in fact transfer Krull.
It turns out that this is not the case: the first Weyl algebras will provide counterexamples in each class.

For Weyl algebras, we have the following very strong obstacle to studying their arithmetic via any commutative cancellative monoid.

\begin{lemma}
  Let $K$ be a field of characteristic not $2$, and let $A\coloneqq A_1(K)\coloneqq K[x][y;\delta]$ with $\delta(x)=1$ be the first Weyl algebra.
  Then there exists no length-preserving monoid homomorphism $\varphi\colon A^\bullet \to D$ to any commutative cancellative monoid $D$.
\end{lemma}

\begin{proof}
  We have $xy - yx = 1$ and therefore
  \[
    x^2y = (1+xy)x.
  \]
  It is clear that $x$,~$y$ are atoms, and a direct computation shows that also $1+xy$ is an atom of $A^\bullet$ (here we use $\operatorname{char}(K)\ne 2$,  otherwise $1+xy = 2+yx=yx$ factors).

  Suppose there exists an length-preserving homomorphism $\varphi\colon A^\bullet \to D$ to a commutative cancellative monoid $D$.
  Then $\varphi(x)^2 \varphi(y) = \varphi(1+xy)\varphi(x)$, and hence $\varphi(1+xy)=\varphi(x)\varphi(y)$.
  By assumption on $\varphi$, the elements $\varphi(x)$ and $\varphi(y)$ are atoms, so in particular non-units, contradicting that $\varphi(1+xy)$ is an atom.
\end{proof}

\begin{example}
  Independent of the base field, the Weyl algebra $A=A_1(K)$ is a noetherian domain of quadratic growth, Auslander-regular \cite{E}, and a maximal order in its simple Artinian ring of quotients \cite[Corollary 5.1.6]{mcconnell-robson01}.
  By the previous lemma, there exists no length-preserving monoid homomorphism to a commutative cancellative monoid.
  \begin{enumerate}
  \item Suppose $\operatorname{char}(K)=0$.
    Then $A_1(K)$ is a simple Dedekind domain \cite[Corollary 7.11.3]{mcconnell-robson01}.
    In particular, it has Krull and global dimension $1$.
    Thus $A_1(K)$ has bounded factorizations by any of \cref{p:smalldim}, \cref{p:filtration}, \cref{cor:iterated}, \cref{thm: growth}, \cref{AG}, \cref{gradejump}.
  \item
    If $K$ has characteristic $p > 0$, $p \ne 2$, then $A_1(K)$ has center $K[x^p,y^p]$ and Krull and global dimension $2$.
    In this case $A_1(K)$ is module-finite over its center and therefore a PI ring.
    Since it is a noetherian maximal order and PI, it is also a bounded Krull order.
    Thus $A_1(K)$ has bounded factorizations by any of \cref{p:filtration}, \cref{p: bounded}, \cref{cor:iterated}, \cref{c:prime-pi}, \cref{AG}, \cref{gradejump}.
  \end{enumerate}
\end{example}

\bibliographystyle{hyperalphaabbr}
\bibliography{bf}

\end{document}